\documentclass[10p,twocolumn]{amsart}

\usepackage[utf8]{inputenc}
\usepackage[english]{babel} %% English standard typographical conventions
\usepackage{graphicx,color,comment}  %% figures, more colors
\usepackage[dvipsnames]{xcolor} %% more colors
\usepackage{geometry,hyperref}
\usepackage{amsmath,amssymb,amsfonts,amsthm} %% ams packages
\usepackage{mathtools,stmaryrd}  %% add-on to amsmath
\definecolor{light-gray}{gray}{0.95}
\usepackage[color=light-gray]{todonotes}
\usepackage{tikz-cd}  %%For commutative diagrams with tikz
\usepackage[all,cmtip]{xy} %% For diagrams
\usepackage{enumerate}
\theoremstyle{definition}
\newtheorem{theorem}{Theorem}[section]
\newtheorem{lemma}[theorem]{Lemma}
\newtheorem{example}[theorem]{Example}
\newtheorem{examples}[theorem]{Examples}

\newtheorem{definition-lemma}[theorem]{Definition-Lemma}
\newtheorem{proposition}[theorem]{Proposition}
\newtheorem{corollary}[theorem]{Corollary}
\newtheorem{remark}[theorem]{Remark}

\def\toto{\rightrightarrows}
\def\xto{\xrightarrow}
\def\xfrom{\xleftarrow}
\def\R{{\mathbb R}}
\def\U{{\mathcal U}}
\def\action{\curvearrowright}
\def\from{\leftarrow}
\def\<{\langle}
\def\>{\rangle}
\def\then{\Rightarrow}
\def\id{{\rm id}}
\def\pr{{\rm pr}}
\def\raction{\curvearrowleft}
\def\X{{\mathfrak X}}
\def\g{{\mathfrak g}}
\def\mon{{\rm Mon}}

\def\blue{\textcolor{blue}}

\title{Lie Groupoids}

\author{Henrique Bursztyn}
\address{Instituto de Matem\'atica Pura e Aplicada,
Estrada Dona Castorina 110, Rio de Janeiro, 22460-320, Brasil }
\email{henrique@impa.br}

\author{Matias del Hoyo}
\address{Universidade Federal Fluminense (UFF), Rua Prof. M. W. de Freitas Reis, s/n, Niterói, 24.210-201 RJ, Brazil.}
\email{mldelhoyo@id.uff.br}

\begin{document}

%\nocite{*}

\maketitle

%%%%%%%%%%%%%%%%%%%%%
\section{Introduction}

% born / history / grow / applications

A Lie groupoid can be thought of as a generalization of a Lie group in which the multiplication is only defined for certain pairs of elements. From another perspective, Lie groupoids can be regarded as manifolds endowed with a type of ``action'' codifying internal and external symmetries. Examples naturally arise from Lie group actions, fibrations, and foliations, so Lie groupoids provide a unifying conceptual framework for these (and other) geometric objects.

$$
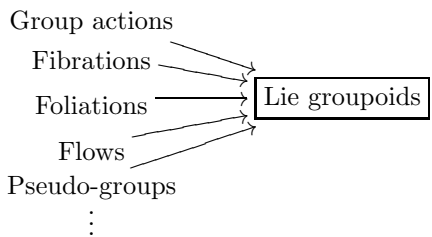
\xymatrix@R=0pt{
  \txt{Group actions} \ar[rdd] &  \\
  \txt{Fibrations} \ar[rd] & \\
 \txt{Foliations} \ar[r] & \fbox{Lie groupoids}   \\
   \txt{Flows} \ar[ru] & \\ 
 \txt{Pseudo-groups} \ar[ruu] & \\
 \vdots &
}$$

Lie groupoids can be traced back to the work of C. Ehresmann in the 1950s \cite{ehresmann}. The vigorous development of their theory in the last few decades has been largely stimulated by their connections with such areas as Poisson geometry and non-commutative geometry \cite{dsw,connes,cdw}, as well as several topics in mathematical physics, including classical mechanics, quantization and topological field theories \cite{catfel,hawkins,we91,we96}. 
Lie groupoids are a rich source of $C^*$-algebras through their convolution algebras; they are also closely related to simplicial manifolds and differentiable stacks, having an important role in the emerging fields of higher and derived differential geometry. 
%\textcolor{blue}{** add REF?}

%Generalizing the correspondence between Lie groups and Lie algebras, Lie groupoids have infinitesimal counterparts known as {\em Lie alge\-broids}, which are the subject of a companion paper \cite{mein} in this volume. 
%$$\xymatrix@R=0pt{
%   & \txt{Lie algebroids} \\
%\fbox{Lie groupoids} \ar[ru] \ar[r] \ar[rd]  & \txt{Stacks}\\
%& \txt{Simplicial manifolds}
%}$$

$$\xymatrix@R=0pt{
   & \txt{Lie algebroids} \\
\fbox{Lie groupoids} \ar[ru] \ar[r] \ar[rd] \ar[rdd]  &   \txt{Stacks} \\
& \txt{$C^*$-algebras}  \\
& \txt{Simplicial manifolds}
}$$

This article is an overview on Lie groupoids, including basic definitions, key examples and constructions, and topics such as actions and representations, local models, Morita equivalence and cohomology. Convolution algebras will not be trated in this paper, see e.g. \cite{connes,dsw}. Much of the material discussed in $\S$ \ref{sec:basic}--$\S$ \ref{section:actions} can be found in more detail in textbooks such as \cite{dsw,macbook,mmbook}.

We will work with $C^\infty$, real manifolds and smooth maps.

%%%%%%%%%%%%%%%%%%%%%%%%%%%%%%%%%%%%%%%%%%

\section{Basic definitions}\label{sec:basic}

% definition

A {\bf Lie groupoid} $G\toto M$ consists of manifolds $M$ (called ``objects'') and $G$ (called ``arrows''), along with structure maps given by surjective submersions $s,t: G \to M$, called {\em source} and {\em target} maps, a smooth partial {\em multiplication} 
$$
m: G \times_M G\to G, \;\; m(h,g)= hg,
$$
where $G\times_MG=\{(h,g):s(h)=t(g)\}$ is the set of pairs of composable arrows (which is an embedded submanifold of $G\times G$ since $s$ and $t$ are submersions), a smooth {\em unit} map $u:M\to G$, $x\mapsto 1_x$, and a smooth {\em inverse} map $i: G\to G$, $g\mapsto g^{-1}$, satisfying group-like axioms described below. It is convenient to think of an element $g\in G$ as an arrow from its source $x=s(g)$ to its target $y=t(g)$, pictured as  $y\xfrom g x$. With this notation, the groupoid axioms are as follows.
%a manifold $M$ of objects, a manifold $G$ of arrows between these objects, whose source and targets are pointed by the surjective submersions $s,t: G \to M$, respectively, and a smooth partial {\bf multiplication} $$m: G \times_M G\to G \quad m(z\xfrom h y,y\xfrom g x)=(z\xfrom{hg}x)$$
%defined over $G\times_MG=\{(h,g):s(h)=t(g)\}$, the pairs of composable arrows, which is an embedded submanifold of the product, for $s$ and $t$ are transverse.
%This multiplication has to satisfy the following group axioms: 
\begin{enumerate}[\ \ $\bullet$]
    \item For $(h,g)\in G\times_M G$,  $s(hg)=s(g)$ and $t(hg)=t(h)$:
    $$
    m(z\xfrom h y,y\xfrom g x)=(z\xfrom{hg}x).
    $$
    \item $m$ is associative, namely $(k h)g=k(hg)$ whenever $k,h,g$ are composable;
    \item For each $x\in M$, $u(x)=1_x$ is an arrow $x\xfrom {1_x}x$ such that $1_{t(g)}g=g=1_{s(g)}$.
    \item For each $y\xfrom g x$ in $G$, $g^{-1}$ is an arrow $x\xfrom {g^{-1}}y$ such that $g^{-1}g=1_x$ and $gg^{-1}=1_y$.
\end{enumerate}  

% remarks on definition

It follows from this definition that $u$ is an embedding, $m$ is a submersion, and $i$ is a diffeomorphism. 
We will usually identify $M$ with the submanifold of $G$ defined by its image under $u$.

In the language of categories and functors, a groupoid is a small category in which every arrow is invertible, and a Lie groupoid is an internal groupoid in the category of smooth manifolds for which source and target maps are required to be submersions. 

Two simple but important examples of Lie groupoids arise when we set objects or arrows to be trivial. Lie groupoids with a single object $G\toto \ast$ are the same as Lie groups. A manifold $M$ gives rise to a Lie groupoid $M\toto M$ with only identity arrows, and this is the canonical way we regard manifolds as Lie grou\-poids. We will see more examples in $\S$ \ref{sec:ex}.

\begin{remark}
It is often necessary to consider Lie groupoids whose manifolds of arrows are non-Hausdorff. We will assume Lie groupoids to be Hausdorff for simplicity, indicating the results and examples where non-Hausdorffness is needed. 
\end{remark}

% morphisms

A {\bf morphism} $\phi:(G\toto M)\to(G'\toto M')$ between Lie groupoids consists of smooth maps $\phi_1:G\to G'$ and $\phi_0:M\to M'$ which are compatible with the source, target, multiplication and units, namely $\phi(1_x)=1_{\phi(x)}$, $s(\phi(g))=\phi(s(g))$, $t(\phi(g))=\phi(t(g))$ and
$\phi(h)\phi(g)=\phi(hg)$. 
In other words, a Lie groupoid morphism is a smooth functor. 
Two Lie groupoid morphisms $\phi,\psi:(G\toto M)\to (G'\toto M')$ are {\bf homotopic} if there is a smooth map $\alpha:M\to G'$ defining a natural isomorphism between $\phi$ and $\psi$.
%, ******* namely $s\alpha=\phi_0$, $t\alpha=\psi_0$, and for every $y\xfrom g x$ in $G$ the square commutes:
%$$\xymatrix{\phi_0(x) \ar[d]_{\phi_1(g)} \ar[r]^{\alpha(x)} & \psi_0(x) \ar[d]^{\psi_1(g)} \\ \phi_0(y) \ar[r]^{\alpha(y)} & \psi_0(y)}$$
%%%%%%%%%%%%%%%%%%%%%%%%%%%%%%%%%%%%%

%Given $G\toto M$ a Lie groupoid and $x,y\in M$, we write $G(-,x)=s^{-1}(x)$, $G(y,-)=t^{-1}(y)$ and $G(y,x)=t^{-1}(y)\cap s^{-1}(x)$. 
A {\bf Lie subgroupoid} $H\toto N$ is a Lie groupoid together with a morphism $H\to G$ that is an injective immersion. The Lie subgroupoid is called {\bf wide} when $M=N$, {\bf full} if $g\in H$ whenever $t(g), s(g) \in N$, 
%$H(y,x)\cong G(y,x)$ for all $y,x\in N$, 
and  {\bf normal} if it is wide, $s(h)=t(h)$ for all $h\in H$, and $ghg^{-1}\in H$ for every $h\in H$ and $g\in G$ for which the composition makes sense.
Given a morphism $\phi: G \to G'$ with $\phi_0=\mathrm{Id}_M$, its kernel $\ker(\phi)=\{ g\in G\,|\, \phi(g) \text{ is a unit} \}$ is a normal subgroupoid of $G$; conversely, every embedded normal Lie subgroupoid is the kernel of such a morphism.

%\blue{relation to kernels of morphisms...}

%later? can mention def of normal subgroupoid, kernels of morphisms over same base, quotients may not be smooth}

% bisections
A Lie groupoid codifies a group formed by its {\em bisections}, as we briefly explain. Given a Lie groupoid $G\toto M$ and an open subset $U\subset M$,  a {\bf local bisection} $\sigma:U\to G$ is a local section for the source map such that the induced map $t\circ \sigma:U\to M$ is an open embedding. Equivalently, by identifying $\sigma$ with its image $B=\sigma(U)$, a local bisection can be defined as an embedded submanifold
$B\subset G$ such that both restrictions $s,t:B\to M$ are open embeddings. When these maps are diffeomorphisms onto $M$, one refers to $B$ as a (global) {\bf bisection}.
% pseudo-group
Setting $V=t(\sigma(U))$, we can visualize a local bisection as a collection of arrows from $U$ to $V$. Local bisections can be composed and inverted, and for every open $U\subset M$ we have the unit local bisection $u(U)$. Hence local bisections give rise to a {\bf pseudogroup} ${\rm Bis}_{loc}(G)$ on $M$. Global bisections form an infinite-dimensional Lie group ${\rm Bis}(G)$  \cite{schmed}.  When $G\toto *$ is a Lie group, the corresponding group of bisections is $G$ itself. For a manifold, the group of bisections is trivial.
Bisections can be used to define (left, right) translations and conjugation on a Lie groupoid. 

%In general, a local bisection $B$ gives rise to diffeomorphisms $L_B: t^{-1}(U) \to t^{-1}(V)$ by left multiplication, $R_B: s^{-1}(V) \to s^{-1}(U)$ by right multiplication, and an isomorphism $C_B:(G_U\toto U)\to (G_V\toto V)$ by conjugation, where $G_U=s^{-1}(U)\cap t^{-1}(U)$.

%local left translations $L_B$ and conjugation $C_B$ diffeomorphisms:
%There are local {\bf translation} and {\bf conjugation} maps associated to a bisection $B$, namely
%$$L_B:G(U,-)\to G(V,-)$$ 
%$$C_B:(G_U\toto U)\to (G_V\toto V)$$
%where $G(U,-)=t^{-1}(U)$ and $G_U=G(U,U)=s^{-1}(U)\cap t^{-1}(U)$. 
%It is clear that $L_B$ is a diffeomorphism, and $C_B$ is in fact a Lie groupoid isomorphism, with inverses $L_{B^{-1}}$ and $C_{B^{-1}}$, respectively. 

% orbits and isotropy

A Lie groupoid defines an equivalence relation on its manifold of objects whose equivalence classes are called {\em orbits}. Given a Lie groupoid $G \toto M$ and $x\in M$, 
the {\bf orbit} of $x$ is the set 
$$
O_x=\{y\,|\, \exists \; y\xfrom g x\}=t(s^{-1}(x)).
$$
The {\bf isotropy} group at $x$ is the group 
$$
G_x=s^{-1}(x)\cap t^{-1}(x)\subseteq G.
$$ 
A Lie groupoid is called {\bf transitive} if it has only one orbit, and it is called {\bf free} if its isotropy groups are all trivial. 
Using local bisection, one can see that, locally, a Lie groupoid looks the same around points of the same orbit.  
Also, if two points $x$, $y$ belong to the same orbit, so that they are connected by an arrow $y\xfrom g x$, then conjugation by $g$ gives an isomorphism between the isotropy groups $G_x\cong G_y$. 

Since the source map $s:G\to M$ is a submersion, for each $x\in M$ the source fiber $s^{-1}(x)$ is an embedded submanifold which projects onto the orbit $O_x$ via the target map, $t:s^{-1}(x) \to O_x$, and carries a right $G_x$-action by right multiplication: $s^{-1}(x) \times G_x\to s^{-1}(x)$,
$$
(y\xfrom{h} x, x\xfrom g x) \mapsto (y \xfrom{hg} x).
$$

\begin{proposition}
Given a Lie groupoid $G\toto M$ and $x\in M$, the following holds.
\begin{enumerate}[\ \ (a)]
    \item The isotropy group $G_x\subseteq G$ is an embedded submanifold and a Lie group.
    \item The action $s^{-1}(x) \times G_x\to s^{-1}(x)$ is smooth, free and proper.
    \item The orbit $O_x$ is an immersed submanifold of $M$ (possibly not embedded) in such a way that $t: s^{-1}(x) {\to} O_x$ is a principal $G_x$-bundle.
    %over $O_x$.
\end{enumerate}
\end{proposition}

% orbit space etc

The {\bf orbit space} $M/G$ of a Lie groupoid $G\toto M$ is the set of orbits with the quotient topology. The quotient map $M\to M/G$ is open, and the partition of $M$ into the connected components of the orbits is a singular foliation, called the {\bf characteristic foliation} of $G$. A Lie grou\-poid is {\bf regular} if so is this foliation, or equivalently if every isotropy group has the same dimension. 
%\textcolor{blue}{move this when defining actions of Lie groupoids? Generalizing the example of action groupoids, a Lie groupoid is {\bf free}, {\bf proper}, or {\bf principal} if its the anchor map $(t,s): G\to M\times M$ is injective, closed *** or embedded, respectively. A free and proper groupoid is principal, and a principal groupoid is free but not necessarily proper. By Godement criterion, the orbit space $M/G$ of a principal groupoid is a manifold, and such an action is proper if and only if $M/G$ is Hausdorff \cite[Appendix]{dhl}. }

% classes of groupoids

%The following are important classes of Lie grou\-poids. 

A Lie groupoid $G\toto M$ is said to be {\bf $s$-connec\-ted} (resp. {\bf $s$-simply connec\-ted}), if the source fibers $s^{-1}(x)$ are connected (resp. connected and simply connected).

%\begin{enumerate}[\ \ $\bullet$]
%    \item {\bf $s$-connected} (resp. {\bf $s$-simply connec\-ted}), if the source fibers $s^{-1}(x)$ are connected (resp. connected and simply connected);
%    \item {\bf étale} if $G$ and $M$ have the same dimension;
%    \item {\bf proper} if the map $(t,s):G\to M\times M$ is proper.
%\end{enumerate}

\section{Lie algebroids and Lie theorems}
% induced Lie algebroid

Lie groupoids have infinitesimal counterparts that extend the way Lie algebras arise from Lie groups.

A {\bf Lie algebroid} consists of a vector bundle $A\to M$ together with a vector-bundle map $\mathtt{a}:A\to TM$, called the {\it anchor map}, and a Lie bracket  $[\cdot,\cdot]$ on $\Gamma(A)$ satisfying the Leibniz rule,
$$
[u,fv]=f[u,v]+ (\mathtt{a}(u) f) v
$$
for all $u,v \in \Gamma(A)$ and $f\in C^ \infty(M)$.
Lie algebroid morphisms are vector bundle maps compatible with the bracket and the anchor, see \cite{macbook}.
We use the notation $A\Rightarrow M$ for a Lie algebroid.

Given a Lie groupoid $G\toto M$, its structure maps give rise to a Lie algebroid $A_G \Rightarrow M$ as follows: 

%We can regard $M$ as an embedded submanifold of both $G$ and $G\times_M G$ via the unit map, and the germs of the other structure-maps around $M$ give rise to the {\bf induced Lie algebroid} $(A_G,[,],dt)$. Concretely, 
\begin{enumerate}[\ \ $\bullet$]
    \item As a vector bundle, $A_G=\ker(ds)|_M\to M$ is the normal bundle of $M\subset G$; so a vector $v\in (A_G)_x$ is tangent to the $s$-fiber at $x$;
    \item The Lie bracket $[\cdot,\cdot]$  on $\Gamma(A_G)$ is defined by identifying the sections of $A_G$ with vector fields $X$ on $G$
    %\in\ker(ds)\subset TG$ 
    that are tangent to $s$-fibers and right-invariant, 
    $$
    dR_h(X_g)=X_{hg},
    $$
    where $R_h: s^{-1}(t(h))\to s^{-1}(s(h))$ is defined by right translation, $R_h(g)=hg$;
    \item The anchor map is the restriction of $dt: TG\to TM$ to $A_G\subseteq TG|_M$, 
    $$
    dt|_{A_G}:A_G\to TM.
    $$ 
    %is the anchor-map, which plays a key role in the following Leibniz rule
   % $$[\alpha,f\beta]=f[\alpha,\beta]+ dt(\alpha)(f)\beta$$
\end{enumerate}

%\begin{remark}
%One can alternatively define $A_G$ as $\mathrm{ker}(dt)|_M$ with bracket given by the identification of $\Gamma(A)$ with left-invariant vector fields, and anchor defined by $-ds$, as in \cite{mein}. 
%The inversion map on $G$ identifies the two conventions. \
%\end{remark}

A Lie algebroid $A\Rightarrow M$ is {\bf integrable} if there is a Lie groupoid $G\toto M$ and an isomorphism $\psi:A_G\to A$; in this case we say that $(G,\psi)$ is an {\bf integration} of $A$. In contrast with Lie algebras, not every Lie algebroid is integrable \cite{almo}.

% abstract Lie algebroid

%An {\bf abstract Lie algebroid} $(A,[,],a)$ consists of a vector bundle $A\to M$, a Lie bracket $[,]$ on $\Gamma(A)$, and an {\it anchor map} $a:A\to TM$ satisfying the Leibniz rule. 

The construction of the Lie algebroid associated with a Lie groupoid is functorial:  any Lie groupoid morphism $\phi: G\to H$ differentiates to a Lie algebroid morphism
$$
\phi': A_G\to A_H.
$$
This gives rise to the so-called {\bf Lie functor} 
$$
\text{Lie groupoids}\to\text{Lie algebroids},
$$
extending the one from Lie groups to Lie algebras.
The interplay between global and infinitesimal objects is governed by the following generalizations of the classical fundamental Lie theorems, see e.g. \cite[$\S$ 6.3]{mmbook}. 
 
%The induced Lie algebroid construction gives rise to the so-called {\bf Lie functor} 
%$$\text{Lie groupoids}\to\text{Lie algebroids},$$
%setting an interplay between global and infinitesimal information, analogous to that of Lie groups and Lie algebras, and ruled by the following Fundamental Lie Theorems. 

\begin{theorem}[Lie 1]\label{thm:Lie1}
Any integrable Lie algebroid $A\then M$ admits a universal integration $(G(A)\toto M,\tilde\psi)$, possibly non-Hausdorff, in the sense that if $(G\toto M,\psi)$ is another integration, then there is a unique morphism $\phi:G(A)\to G$ such that $\psi\phi'=\tilde\psi$.
\end{theorem}

%One refers to $G(A)$ as the {\bf universal integration} of $A$;
The integration $G(A)$ is unique, up to isomorphism, and characterized by the property of being $s$-simply connected. Indeed, one can build $G(A)$ by assembling the universal covering spaces of the $s$-fibers of any given integration $G$. 
%\textcolor{blue}{repeated below...:In fact, one can build $G(A)$  by assembling the universal covers of the restrictions of any given integration $G$ to the orbits: later: }. 
Moreover, $G(A)$ is Hausdorff if and only if the foliation by $s$-fibers has no vanishing cycles.

%It is not hard to see that the universal integration $G(A)$ is characterized by having simply connected $s$-fibers, and in fact, starting from any integration $G$, we can recover $G(A_G)$ by collecting the universal covers of the restriction to the orbits. Moreover, $G(A)$ is Hausdorff if and only if the foliation by $s$-fibers has no vanishing cycles.

The next result concerns the integration of morphisms.

\begin{theorem}[Lie 2]
Let  $G\toto M$ be an s-simply-connected Lie groupoid, $H\toto N$ a Lie groupoid, and $\psi:A_G\to A_H$ a Lie algebroid morphism. Then there exists a unique morphism $\phi:G\to H$ such that $\phi'=\psi$.   
\end{theorem}

Just as Lie 1, the previous theorem has a proof analogous to the case of Lie groups and Lie algebras: one finds (the graph of) $\phi$ by integrating the subalgebroid of $A_G\times A_H$ given by the graph of $\psi$ to a subgroupoid of $G\times H$.

A consequence of Lie 1 and Lie 2 is that the Lie functor establishes an equivalence 
 between the category of $s$-simply connected Lie groupoids and the category of integrable Lie algebroids.

%**** These two results have proofs analogous to those for Lie groups and Lie algebras. Lie 1 can be proven by showing that the universal covering spaces of the $s$-fibers fit into a well-defined Lie groupoid $G(A_G)$, and Lie 2 follows by integrating the algebroid-theoretic graph of $\psi$ to a subgroupoid of $G\times H$. As a corollary, they yield an equivalence between the categories of $s$-simply connected Lie groupoids and integrable Lie algebroids. 

As previously mentioned, not every Lie algebroid comes from a Lie groupoid, so the direct generalization of Lie 3 (asserting that every Lie algebra is integrable) does not hold. A fundamental result in \cite{cf} uses infinite-dimensional techniques (the path-space construction of the so-called {\em Weinstein groupoid}, see also \cite{sev}) to identify the precise obstructions to integrability, given in terms of the uniform discreteness of certain ``monodromy groups''; see \cite{mein} for a further discussion.

One can also consider weaker forms of integration, e.g. by {\em local} Lie groupoids \cite{cms} (see \cite{femi} for the passage from local to global), 
or by allowing non-smooth global objects \cite{tz}.

\section{Examples and basic constructions}\label{sec:ex}

Lie groupoids unify several classical geometrical objects, which can be used to illustrate different aspects of the general framework and to which the general theory can be applied.

We already mentioned that Lie groups are identified with Lie groupoids over a single object, and manifolds can be regarded as Lie groupoids with only identity arrows.
We now list several other fundamental examples.

%**** Lie groupoids provide a unifying framework to deal with several classic geometries, each of which provides some intuition for Lie groupoids, and is a field of application of the results on them. Some of the fundamental examples are the following:

%**** point out group of global bisections in some cases?

%\begin{example} % Lie group and manifolds
%Lie groupoids have two levels of structure, and two simple but important examples appear when one of these levels is trivial.
%A {\bf Lie group} $G$ can be seen as a Lie groupoid with a single object $G\toto \ast$, and we can actually identify Lie groups with Lie groupoids with a single object. The induced Lie algebroid is the usual Lie algebra, and it is always integrable.
%A manifold $M$ gives rise to a {\bf unit groupoid} $M\toto M$ with only identity arrows, and this is the canonical way we will regard manifolds as groupoids. The induced Lie algebroid is the 0-dimensional vector bundle $0_M\then M$.
%\end{example}

\begin{example} %pair
Given a manifold $M$, its {\bf pair groupoid} $M\times M \toto M$ has exactly one arrow between any two objects, $y\xfrom{(y,x)}x$, and multiplication
$$
(z,y)(y,x)=(z,x).
$$
This is a transitive Lie groupoid with trivial isotro\-pies. Its group of (global) bisections is the diffeomorphism group of $M$. Note that, for any Lie groupoid $G\toto M$, there is a natural groupoid morphism $(t,s): G\to M\times M$. The associated Lie algebroid is $TM \Rightarrow M$, with the usual Lie bracket of vector fields and the identity map as the anchor.
\end{example}

\begin{example}\label{ex:submersion}% submersion groupoid
A surjective submersion $\tau:M\to N$ yields a {\bf submersion groupoid} $$
M\times_NM\toto M,
$$
which is a wide Lie subgroupoid of $M\times M$ with one arrow between two objects if and only if they belong to the same fiber. The isotropies are trivial and the orbits are the fibers of $\tau$, so it is possible to recover $N$, up to diffeomorphism,
as the orbit space of $M\times_N M$. The associated
Lie algebroid is given by the involutive distribution tangent to the $\tau$-fibers. When $\tau:M\to \ast$ is the projection to a point, we just obtain the pair groupoid. An interesting example is when $\tau: \U=\coprod_i U_i\to M$ is an open cover, in which case we obtain the {\bf Cech groupoid} 
\begin{equation}\label{eq:cech}
G_\U= \coprod_{ji} U_{ji}\toto\coprod_i U_i,
\end{equation}
with objects the pairs $(x,i)$ with $x\in U_i$, and arrows $(x,j)\from(x,i)$ with $x\in U_{ji} = U_j\cap U_i$.
\end{example}

\begin{example}\label{ex:fundgrp}% fundamental groupoid
Given $M$ a manifold, its {\bf fundamental groupoid} $\Pi_1(M)\toto M$
is a Lie groupoid whose arrows are given by homotopy classes of paths on $M$,
$$
\gamma(1) \xfrom{[\gamma]} \gamma(0),
$$
and multiplication defined by concatenation. 
%\textcolor{blue}{chart omitted}
%(A chart $V\times U$ of $\Pi_1(M)$ around $y\xfrom{[\gamma]}x$ can be obtained as the product of charts $V,U$ of $M$ centered at $y$ and $x$, mapping a pair $(y',x')$ to the (homotopy class of) the concatenation of $\gamma$ with the chart segments $x-x'$ and $y-y'$.)
The orbits of $\Pi_1(M)$ are the connected components of $M$, and its isotropy group at $x$ is the fundamental group $\pi_1(M,x)$ of the connected component containing $x$. For $M$ connected, $\Pi_1(M)$ is transitive and each $s$-fiber is identified with the universal cover of $M$ as a principal $\pi_1(M)$-bundle. When $M$ is connected and simply connected, $\Pi_1(M)$ coincides with the pair groupoid $M\times M$.
The fundamental groupoid $\Pi_1(M)$ is the universal integration of $TM\Rightarrow M$, in the sense of Lie 1 (Theorem~\ref{thm:Lie1}).
\end{example}

%If an equivalence relation $R\subset M\times M$ is a closed embedded submanifold such that $\pi_1:R\to M$ is a submersion, then $R\toto M$ is a Lie groupoid with trivial isotropies, and by Godement's criterion,  the orbit space $M/R$ inherits a natural smooth structure so that the projection is a submersion.

\begin{example}
A {\bf bundle of Lie groups} $\{G_x\}_{x\in M}$ is the same as a Lie groupoid $G\toto M$ 
whose source and target maps coincide.
As a particular case, any vector bundle $E\to M$ can be regarded as a bundle of abelian Lie groups, where the multiplication is given by fiber-wise addition. Bundles of Lie groups need not be locally trivial.
An example is $\{\R^2_\epsilon\}_{\epsilon\in\R}$, where the product in $\R^2_\epsilon$ is given by $(x,y)\cdot_\epsilon(x',y')=(x+x',y+e^{x\epsilon}y')$.
The infinitesimal counterpart of a bundle of Lie groups is a bundle of Lie algebras, i.e., a Lie algebroid with trivial anchor ($\mathtt{a}=0$), in such a way that each fiber is a Lie algebra.
\end{example}

\begin{example}\label{ex:action-gpd} % action groupoid
An action  $K\action M$ of a Lie group $K$ on a manifold $M$ gives rise to an {\bf action groupoid} $K\times M\toto M$, with arrows $k.x \xfrom{(k,x)} x$ and multiplication given by
$$
(l,y)(k,x)=(lk,x).
$$
%source the projection, target the action map $\rho$, and multiplication, unit, and inverses coming from that of $K$. 
Action groupoids are written as $K\ltimes M$.
The orbits and isotropy groups of the action grou\-poid coincide with those of the action.
Note that it is not always possible to recover the group action from the action groupoid:
e.g., when the $K$-action on $M$ is free and transitive, the action groupoid is identified with the pair groupoid $M\times M$. The Lie algebroid of the action groupoid is the so-called {\em action Lie algebroid} corresponding to the infinitesimal action.
%Thus, it is not always possible to recover the group action out of the action groupoid. The induced Lie algebroid $({\mathfrak g}_M,[,],\rho')$ encodes the infinitesimal action $\rho':\mathfrak g\to{\mathfrak X}(M)$ induced by $\rho$. %and since the same manifold may admit more than one Lie group structure, we cannot recover the whole action out of its associated groupoid, unless further hypothesis is assumed.
\end{example}

\begin{example} % flows
Given a complete vector field $X$  on a manifold $M$, its flow defines an action $\R\curvearrowright M$, which yields an action groupoid $\R\times M\toto M$. When $X$ is not complete we do not have an $\R$-action, but we can still find a maximal open subset $U_X\subseteq \R\times M$ where its flow is defined, 
$$
\varphi_X: U_X\to M, \; (t,x)\mapsto \varphi_X^t(x).
$$
The flow of $X$ gives rise to the {\bf flow groupoid} $U_X\toto M$, with arrows 
$\varphi_X^t(x) \xfrom{(t,x)} x$ and multiplication $(s,y)(t,x)=(s+t,x)$. The orbits of the flow groupoid are the integral curves of $X$. The associated Lie algebroid is $\R\times M\to M$ with anchor map given by $1\mapsto X$ (which uniquely determines the bracket). Integration of vector fields sets a bijective correspondence between vector fields and flow groupoids.
%****** we can still define a {\bf flow groupoid} $U_X\subset\R\times M\toto M$, where $U_X$ is the open subset of pairs $(t,x)$ such that it is possible to flow time $t$ from $x$ with $X$. The flow groupoid is $s$-simply connected and its induced Lie algebroid is $(\R_M,0,X)$. Integration sets a 1-1 correspondence between vector fields and (maximal) flow groupoids. 
% which are maximal with respect to the inclusion.
\end{example}

\begin{example}\label{ex:GLE}
Generalizing the general linear group of a vector space, any vector bundle $E\to M$ has a corresponding {\bf general linear grou\-poid} $\mathrm{GL}(E) \toto M$ with arrows $y \xfrom{} x$ given by linear isomorphisms $T: E_x\to E_y$, and multiplication defined by composition. This is a transitive Lie groupoid with isotropy groups given by the general linear groups of the fibers. The group of (global) bisections is the group of vector-bundle automorphisms of $E$. The Lie algebroid of $\mathrm{GL}(E)$ has as space of sections the Lie algebra of {\em derivations} of $E$.
\end{example}

\begin{example}\label{ex:gauge}
Let $K$ be a Lie group and $\tau:P\to M$ be a principal $K$-bundle.
%so $P$ carries a free (right) action of $K$ whose orbits are the fibers of $\tau$. 
The {\bf gauge groupoid} $P\times_K P\toto M$ has as arrows the $K$-equivariant maps $\tau^{-1}(x)\to \tau^{-1}(y)$. Alternatively, one can regard $P\times_K P$ as the quotient of the pair groupoid $P\times P$ by the diagonal $K$-action, noticing that the $K$-orbit of $(q,p)$, with $\tau(q)=y$ and $\tau(p)=x$,  can be thought of as (the graph of) a $K$-equivariant map  
from $\tau^{-1}(x)$ to $\tau^{-1}(y)$. The gauge groupoid is a transitive Lie groupoid with isotropy groups (noncanonically) isomorphic to $K$. The group of bisections of $P\times_K P$ is the group of $K$-invariant diffeomorphisms of $P$. The associated Lie algebroid is the so-called {\em Atiyah algebroid} of $P$, whose sections are the $K$-invariant vector fields on $P$. 

%It is easy to see that the gauge groupoid is a well-defined transitive Lie groupoid.
Every transitive groupoid is isomorphic to the gauge groupoid of the principal $G_x$-bundle $t:s^{-1}(x)\to M$, for any $x\in M$. So we can revisit Examples \ref{ex:fundgrp}
and \ref{ex:GLE} from this perspective:
\begin{enumerate}[\ \ $\bullet$]
\item When $M$ is connected, its fundamental grou\-poid $\Pi_1(M)$ is identified with the gauge grou\-poid of the principal $\pi_1(M)$-bundle $\tau:\tilde{M}\to M$ defined by its universal cover.

\item The general linear groupoid $\mathrm{GL}(E)$ of a vector bundle $E\to M$ is identified with the gauge groupoid of the frame bundle $\mathrm{Fr}(E)\to M$.
\end{enumerate}
%When $M$ is connected and $q:\tilde M\to M$ is the projection from the universal cover the gauge groupoid identifies with the fundamental groupoid, and when $q:Fr(V)\to M$ is the frame bundle of a vector bundle, then the gauge groupoid identifies with the {\bf general linear groupoid} $GL(V)\toto M$, with the fibers $V_x$ as objects and the linear isomorphisms $V_x\to V_y$ as arrows. It turns out that every transitive groupoid is isomorphic to the gauge groupoid of $t:G(-,x)\to M$. 
The correspondence between principal bundles and transitive groupoids relies on the choice of a point and is not an equivalence of categories, as easily seen when $M$ is a point.

%The induced Lie algebroid is the so-called {\bf Atiyah algebroid} $(A[P],[,],dq)$, whose sections are $K$-invariant vector fields on $P$.
\end{example}

\begin{example}
Consider a foliation on $M$, defined by an involutive subbundle $F\subseteq TM$.
There are two Lie groupoids naturally associated with it.
The {\bf monodromy groupoid} 
$$
\mathrm{Mon}(F)\toto M
$$ 
has arrows given by  homotopy classes of paths within a leaf, and the {\bf holonomy groupoid} 
$$
\mathrm{Hol}(F)\toto M
$$ 
is the quotient of $\mathrm{Mon}(F)\toto M$ where two paths $\gamma$ and $\sigma$ with the same endpoints are identified if $\sigma\gamma^{-1}$ has trivial holonomy (meaning that, after covering this path with a finite sequence of foliated charts, the composition of the transition maps is the identity). 
The groupoid multiplication in both cases is given by concatenation, and the orbits of both groupoids coincide with the leaves of $F$.
One can use foliated charts to build manifold charts for $\mathrm{Mon}(F)$ and $\mathrm{Hol}(F)$, but the resulting topologies may be non-Hausdorff. In fact, $\mathrm{Mon}(F)$ is non-Hausdorff if and only if $F$ has vanishing cycles, see e.g. {\cite{dhl}}. 

Both Lie groupoids $\mathrm{Mon}(F)$ and $\mathrm{Hol}(F)$ are integrations of $F$, regarded as a Lie subalgebroid of $TM\Rightarrow M$, the monodromy groupoid being the universal one (while the holonomy groupoid is the ``smallest'' integration \cite{crmo}). 
%A foliation $F\subset TM$ can be regarded as a Lie algebroid $(F,[,],i)$ over $M$ and it admits two canonical integrations.
\end{example}

\begin{example}\label{ex:orbifold}
We refer to \cite{mopr} for basic definitions of orbifolds. Given  an orbifold $Q$ and  an orbifold atlas $\U=\{(U_i,G_i,\phi_i)\}$, then its {\bf Cech groupoid}
$$
G_\U\toto \coprod_i U_i
$$
is the colimit of the action groupoids $G_i\times U_i\toto U_i$. Its arrows are germs spanned by (i) maps in some $G_i$, and (ii) orbifold chart embeddings $U_{ij}\to U_j$, and they are endowed with the sheaf-like manifold structure.
When every $G_i$ is trivial, so that $Q$ is a manifold, this recovers the Cech groupoid of Example \ref{ex:submersion}. 
\end{example}

\begin{example}\label{ex:tancot}
There are various Lie groupoids naturally built from a given one.
For a Lie grou\-poid $G\toto M$, its tangent and cotangent bundles have natural Lie groupoid structures,
$$
TG \toto TM, \qquad T^*G\toto A^*_G.
$$
For the tangent bundle, the structure maps are obtained by differentiating the structure maps of $G$. On the cotangent bundle, they can be described by a general duality principle \cite[$\S$ 11.2]{macbook}. When $G\toto \ast$ is a  Lie group, $TG= G\times \g \toto \ast$ is the semi-direct product Lie group with respect to the adjoint representation, while $T^*G= G\times \g^* \toto \g^*$ is the action groupoid with respect to the coadjoint action.
\end{example}

%\begin{example}
%Let $(M,\pi)$ be a Poisson manifold, and let $q:(E,\omega)\to (M,\pi)$ be a {\it complete symplectic realization}, namely (i) $q^*:C^\infty(M)\to C^\infty(E)$ preserves the Poisson brackets, and (ii) the Hamiltonian $X_{fq}$ is complete whenever $X_f$ is so. Then  
%\end{example}

%%%%%%%%%%%%%%%%%%%%%%%%%%%%%%%%%%%%%%%%%%

%**** tangent, cotangent, jets of bisections...

Other examples arise from functorial constructions with Lie groupoids, some of which we mention now. 

By considering $k$-jets of local bisections of $G$, one obtains the so-called {\em jet groupoids} $J^kG \toto M$, that arise e.g. in the theory of Lie pseudo\-groups \cite{cy}.

%\textcolor{blue}{REF, connection to Lie pseudogroups, geometry of PDEs}.

% products
Given two Lie groupoids $G\toto M$ and $G'\toto M'$, their {\bf product} $G\times G'\toto M\times M'$ is naturally a Lie groupoid, with source, target, multiplication and unit defined componentwise. More generally, given another Lie groupoid $H\toto N$ and groupoid morphisms $\phi:G\to H$ and $\psi:G'\to H$, their {\bf fibered product} $G\times_H G'\toto M\times_N M'$ is a well-defined Lie groupoid whenever $\phi$ and $\psi$ are transversal maps, or, in more generality, whenever their fibered product as manifolds exists and is ``good''  (in the sense that it is an embedded submanifold of $G\times G'$ with the expected tangent spaces  \cite{dH}), see \cite[App.~A]{bcdh}. 
%An important variant of this construction, called the {\em homotopy} or {\em weak fibered product}, will be recalled in $\S$ \ref{sec:morita}.

%\textcolor{blue}{homotopy fibered product moved later}

% homotopy fibred product
%An important variant of this construction is the {\bf homotopy fibered product} $G\tilde{\times}_H G'$, which can be computed as the fibered product between $\phi\times\psi$ and the source-target morphism $(t,s):G^I\to G$, and that fills a universal diagram that is commutative up to homotopy.

%Alternatively, $G\times_H G'$ can be obtained as the fibered product between $\phi\times\psi$ and the diagonal map $\Delta:H\to H\times H$. 

% pullback and restriction
If $G\toto M$ is a Lie groupoid and $\tau:N\to M$ is a smooth map, we can consider the fibered product of the morphisms $(s,t): G\to M\times M$ and $(\tau\times \tau): N\times N \to M\times M$. Whenever this results in a Lie groupoid, one refers to it as the {\bf pullback groupoid}, denoted by $\tau^*G\toto N$. This happens for instance when $\tau$ is {transverse to the orbits}, or when $N\subseteq M$ is either an open or {an invariant} submanifold (i.e., a union of orbits), in which case the pullback is called the {\bf restriction} $G_N\toto N$. Every full Lie subgroupoid is in fact a restriction.

% quotients

Let $G\toto M$ be a Lie groupoid and $K$ be a Lie group acting on $M$ and $G$ in such a way that the structure maps are equivariant.
%$K\action G$ and $K\action M$ free proper actions so that the structure maps are equivariant. 
If the $K$-action on $M$ is principal (so that $M\to M/K$ is a principal $K$-bundle), then so is the $K$-action on $G$, in which case the quotient $G/K$ is a well-defined Lie groupoid over $M/K$, called the {\bf quotient groupoid} \cite[Lem.~5.9]{mmbook}.
%$G/K\toto M/K$ is a well-defined Lie groupoid., where 
The groupoid multiplication on the quotient relies on the identification of $G/K\times_{M/K} G/K$ with $G\times_M G/K\times K$. For example, the gauge groupoid of a principal $K$-bundle $P\to M$ is the quotient of the pair groupoid $P\times P\toto P$ by the diagonal $K$-action. 
%****Other type of quotients is when $(I\toto M)\subset (G\toto M)$ is a closed wide intransitive normal subgroupoid, then $G/I\toto M$ is a Lie groupoid.
More generally, one can consider quotients of a Lie groupoid with respect to the action of another Lie groupoid \cite{mmbook}, see the next section.

Other important constructions, such as blow-ups and deformations to the normal cone, are discussed e.g. in \cite{desk}. For further connections with classical differential geometry, see \cite{fest}.

%\textcolor{blue}{relate to section on groupoid actions?}

%%%%%%%%%%%%%%%%%%%%%%%%%%%%%%%%%%%%%%%%

\section{Actions and representations}
\label{section:actions}

% actions and representations

Let $G\toto M$ be a Lie groupoid, $\mu: S\to M$ a smooth map,
%$q: E\to M$ a surjective submersion 
and write $G\times_M S$ for the fibered product over $s$ and $\mu$. A {\bf left action} $$
(G\toto M)\action (\mu:S\to M)
$$ 
is given by a smooth map $\rho:G\times_M S\to S$,
$(g,y)\mapsto g y = \rho_g(y)$, such that 
\begin{enumerate}[\ \ $\bullet$]
\item $\mu(\rho_g(y))= t(g)$, i.e., $\rho_g$ maps $\mu^{-1}(s(g))$ to $\mu^{-1}(t(g))$, 
%(i.e., for $y\xfrom gx\in G$, $\rho_g$ maps $E_x$ to $E_y$),  %get a diffeomorphism $\rho_g:E_x\to E_y$ satisfying 
\item $\rho_{1_x}=\id_{\mu^{-1}(x)}$, for every $x \in M$, 
\item $\rho_{hg}=\rho_h\rho_g$, for every $h,g$ composable. 
\end{enumerate}
The map $\mu$ is often called the {\bf moment map}. Note that when $\mu$ is a submersion, it defines a family of manifolds $\{S_x\}_{x\in M}$ parameterized by $M$, and the action realizes the arrows of $G$ as symmetries of this family. An action of $G\toto M$ on $\mu: S\to M$ also induces a natural action of the group of bisections ${\rm Bis}(G)$ on $S$.
Right actions are defined analogously. Any Lie groupoid $G\toto M$ acts on itself by left multiplication with moment map $\mu=t$, and it acts on $M$ with moment map $\mathrm{Id}_M:M \to M$ via $t: G\times_M M=G \to M$. There are analogous actions induced by right multiplication.
%\textcolor{blue}{action on $id: M\to M$ by $t:G\times_M M=G\to M$}

If $\mu: E\to M$ is a vector bundle, an action is called a {\bf representation} if $\rho_g: E_{s(g)}\to E_{t(g)}$ is linear for every
$g\in G$.
%$y\xfrom g x$.
A representation can be equivalently described as a morphism 
$$
(G\toto M)\to(GL(E)\toto M)
$$ 
into the general linear groupoid of $E$ which is the identity map on objects.
%$\rho^\#_0=\id_M$. 

%\begin{remark} If we relax the definition by merely requiring $q$ to have a constant rank, then $q: E\to S=q(E)$ is a submersion onto an invariant submanifold, and the action is just the same as one of the restriction $G_S\action E$.
%\end{remark}

% example

\begin{examples}
\begin{enumerate}[(i)]
\item 
Actions and representa\-ti\-ons of a Lie group, viewed as a Lie groupoid with a single object, are the usual ones. 
%, and the translation groupoid recovers that of Example \ref{ex:action-gpd};   
\item An action of a manifold $M\toto M$ is just a 
%surjective submersion 
smooth map into $M$, while a representation is just a vector bundle $E\to M$.
\item 
A representation of the pair groupoid 
$$
(M\times M\toto M)\action (E\to M)
$$ 
is the same as a trivialization of $E$, and a representation of the fundamental groupoid $\Pi_1(M)$ on $E$ is the same as a flat connection $\nabla$. In fact, by van Kampen's theorem, $\Pi_1(M)$ is the colimit of the pair groupoids of open charts, from where a representation is the same as locally consistent trivializations.

\item Actions of an action groupoid $K\ltimes M \toto M$ are the same as $K$-equivariant maps $S\to M$. Representations of $K\ltimes M$
are the same as $K$-equivariant vector bundles over $M$.
%and representations of an action groupoid $K\ltimes M \toto M$ are the same as  $K$-equivariant maps into $M$ and $K$-equivariant vector bundles $E\to M$, respectively. 

\item Representations of submersion groupoids are the same as {\em descent datum} for vector bundles. More explicitly, let $G=M\times_N M \toto M$ be the submersion groupoid of $\tau: M\to N$. For any vector bundle $E'\to N$,  
the pullback $\tau^*E' \to M$ carries a natural representation of $G$, and any representation of $G$ turns out to be of this form. Indeed, a representation of $G$ on $E\to M$ defines an equivalence relation on $E$ whose quotient is a vector bundle $E'\to N$ such that $\tau^*E'$ is identified with $E$ (as $G$-representations), see \cite[Thm.~2.1.2]{macbook}. 
 When $M=\coprod_i U_i \to N$ is a covering, a representation of the Cech groupoid \eqref{eq:cech} is the same as a {\em cocycle}, that allows vector bundles $E_i\to U_i$ to be ``glued'' into a vector bundle $E'\to N$. 
%\textcolor{blue}{ok?}

%a representation of the submersion groupoid $G=M\times_NM \toto M$ on $E\to M$ defines an equivalence relation on $E$ such that $E/\sim$ is a vector bundle

%***** If $\U$ is an open cover of $M$, or more generally any submersion groupoid, then an action $\U\action E$ is the same as a cocycle or descent datum, setting an equivalence relation $R\subset E\times E$ making $E$ the pullback of $E/R\to M$ via $\U\to M$. 
\end{enumerate}
\end{examples}

% translation groupoid

Extending Example~\ref{ex:action-gpd}, given a left action of $G\toto M$ on $\mu: S\to M$, one defines the corresponding {\bf action groupoid} 
$$
G\ltimes S = (G\times_M S \toto S)
$$ 
as before, with source and target maps given by $s(g,y) = y$ and $t(g,y) = gy$, and composition
$(h,gy)(g,y)=(hg,y)$.
%and units inherited from those of $G$.  
The action groupoid of a right action is defined analogously. 
%We define the {\bf orbit space} $S/G$ of an action as the orbit space of the corresponding action groupoid $G\ltimes S$.

A {\bf fibration} over $G\toto M$ is a Lie groupoid morphism
$$
\phi:(G'\toto M')\to (G\toto M)
$$ 
such that $\phi_0:M'\to M$ and $\hat\phi=(\phi_1,s'):G'\to G\times_M M'$ are surjective submersions. A fibration $\phi$ has trivial kernel (i.e., only consisting of units) if and only if $\hat\phi$ is a diffeomorphism. Note that, for an action of $G\toto M$ on $\mu: S\to M$, 
the canonical projection $G\ltimes S \to G$ is a fibration (with trivial kernel) whenever $\mu$ is a surjective submersion.

An important type of fibration is given as follows.
A {\bf VB-groupoid} is a morphism
$$
\phi:(\Gamma\toto E)\to (G\toto M)
$$
such that $\phi_1: \Gamma\to G$ and $\phi_0: E\to M$ are vector bundles and the structure maps of $\Gamma$ are vector bundle morphisms. Every VB-groupoid is a fibration. The tangent and cotangent groupoids of Example~\ref{ex:tancot} are examples of VB-groupoids.
The {\bf core} of a VB-groupoid $\phi:\Gamma\to G$ is the vector bundle $C = \ker(s_{\Gamma}|_M: \Gamma|_M\to E)$, and it fits into the {\em core exact sequence} 
$$
0\to t^*C \to \Gamma \stackrel{\hat{\phi}}{\to} s^*E \to 0.
$$
%\textcolor{blue}{write core sequence to conclude the next sentence?}
So a VB-groupoid has trivial kernel (as a fibration) if and only if the core bundle is trivial.
%$C=0_M \to M$. 

The next result relates actions to fibrations, and representations to VB-groupoids. %\textcolor{blue}{REF? thm?}.

% correspondence

\begin{proposition}\label{prop:rep}
(a) The action groupoid construction gives a bijective correspondence between $G$-actions on surjective submersions and fibrations over $G\toto M$ with trivial kernel.
%$\phi: (G'\toto M')\to (G\toto M)$ with trivial kernel. 
(b) Under this bijection, $G$-representations corres\-pond to VB-groupoids over $G$ with  trivial core. 
\end{proposition}

% ruth

As indicated by the previous result, fibrations can be thought of as generalized actions, and VB-groupoids as generalized representations. We will make the latter idea more precise, following \cite{arcr,gsm}.

%A general fibration can be thought of as a {\em generalized action}. This idea has a nice incarnation in the linear setting. 

Given  a Lie groupoid $G\toto M$ and a graded vector bundle $E=E_1\oplus E_0\to M$, a (2-term) {\bf representation up to homotopy} $G\action E$ consists of 
\begin{enumerate}[\ \ $\bullet$]
    \item a linear map $\partial: E_1\to E_0$ over $\mathrm{Id}_M$,
    %a linear map $(E_1)_x\xto{\rho^x} E_0^x$ over each fiber,
    \item a chain map $\rho_g:(E,\partial)|_x\to (E,\partial)|_y$ for each $y\xfrom g x$, and
    \item a chain homotopy $\gamma_{h,g}:\rho_{hg}\to\rho_h\rho_g$ for each pair of composable arrows $h,g$,
\end{enumerate}   
so that $\rho_{1_x}=\id_{E_x}$ for every $x \in M$, $\gamma_{h,g}=0$ whenever $h$ or $g$ is a unit, and the ``Bianchi-type'' identity 
$\rho_k\gamma_{h,g}-\gamma_{k,h}\rho_g=\gamma_{kh,g}-\gamma_{k,hg}$ holds for every triple of composable arrows. %An interpretation of representations up to homotopy as morphisms into a higher version of general linear groupoid was developed in \cite{dhs}.
When $E_1=0$ this recovers the usual notion of representation. 

Every representation up to homotopy $G\action E$ yields a VB-groupoid $t^*E_1\oplus s^*E_0 \toto E_0$ via the so-called {\em Grothendieck construction}, a generalization of the construction of action groupoids. 
%that can be seen as a twisted semi-direct product. 
Conversely, given a VB-groupoid $\phi:(\Gamma\toto E)\to(G\toto M)$, 
by choosing a liner section $\sigma$ of $\hat\phi:\Gamma\to s^*E$, known as a {\em cleavage}, we obtain a representation up to homotopy on $C\oplus E$, with $\partial$ given by the {\em core complex} 
$$t_{\Gamma}:C\to E.$$
%by pushing forward objects and arrows via $\sigma$. 
%Such a $\sigma$ is called a (linear) {\bf cleavage}, and 
Two different cleavages give rise to isomorphic representations up to homotopy. This leads to the following extension of Prop.~\ref{prop:rep} (b)  \cite{gsm}  (see also \cite{dho}). 

%\textcolor{blue}{promote to Theorem?}

\begin{theorem}\label{thm:ruthVB}
The Grothendieck construction sets a one-to-one correspondence between (isomorphism classes of) representations up to homotopy $G\action E_1\oplus E_0$ and VB-groupoid $(\Gamma\toto E_0)\to (G\toto M)$ with core $E_1$.     
\end{theorem}

% example

A central aspect of VB-groupoids is their duality theory.
%\textcolor{blue}{Let $\R_G$ denote the VB-groupoid arising from the trivial representation $G\action \R_M$}. 
Every VB-groupoid $\Gamma\toto E$ over $G\toto M$ has a {\bf dual VB-groupoid} 
$$
(\Gamma^*\toto C^*) \to (G\toto M),
$$ 
where $C$ is the core of $\Gamma$. 
%\textcolor{blue}{did not understand:... whose objects are thelinear morphisms $\Gamma \to\R_G$, and arrows are the linear isomorphisms?}. 
The core of the dual VB-groupoid is $E^*$.

\begin{example}
Following Example~\ref{ex:tancot},
the tangent groupoid $TG\to G$ is a VB-groupoid, with core $A_G\to M$. The 
choice of a linear section $s^*TM\to TG$ over $G$ gives rise to a representation up to homotopy on $A_G\oplus TM$, with core complex 
$$
A_G \stackrel{\rho}{\to} TM.
$$ 
This is the {\bf adjoint representation} of $G$ (well defined up to isomorphism). The cotangent grou\-poid $T^*G\toto A_G^*$ is the dual VB-groupoid and encodes the {\bf coadjoint representation}, with core complex $T^*M \stackrel{\rho^*}{\to} A_G^*$. 
%When $K$ is just a group, the cotangent groupoid identifies with $K\ltimes\g^*\toto \g^*$, the action groupoid of the coadjoint representation.
\end{example}

% The fibration viewpoint gives in particular a way to identify left and right actions. Given $\rho':G\action E'$ and $\phi:G\action E$, the previous lemma also allows us to identify a $G$-equivariant map $\phi:E'\to E$ with an action of the translation groupoid $(G\ltimes E)\action E'$. 

%%%%%%%%%%%%%%%%%%%%%%%%%%%%%%%%%

% free, proper, principal

%\textcolor{blue}{skipped  free, proper, principal actions}

%\textcolor{blue}{ Generalizing the example of action groupoids, a Lie groupoid is {\bf free}, {\bf proper}, or {\bf principal} if its the anchor map $(t,s): G\to M\times M$ is injective, closed  or embedded, respectively. A free and proper groupoid is principal, and a principal groupoid is free but not necessarily proper. By Godement criterion, the orbit space $M/G$ of a principal groupoid is a manifold, and such an action is proper if and only if $M/G$ is Hausdorff \cite[Appendix]{dhl}. }

%A (right) Lie groupoid action $E\curvearrowleft G:\rho$ is {\bf free}, {\bf proper}, or {\bf principal} if the action groupoid $G(\rho)=(E\times_MG\toto E)$ is so. If a Lie groupoid is either free, proper or principal, then every action of it is so, see \cite{dhl}. 
%*** The {\bf orbit space} of the action $E/G$ is that of the translation groupoid $G(\rho)$, it is a manifold if the action is free and proper, and it is a non-Hausdorff manifold iff the action is principal. 

We now discuss principal bundles in the realm of Lie groupoids. Given an action of $G\toto M$ on $\mu: S\to M$, the {\bf orbits} and {\bf orbit space} $S/G$ of the action
are defined as those of the corresponding action groupoid.
A (right) {\bf groupoid bundle} consists of a (right) groupoid action %$S\curvearrowleft G$ 
of $G$ on $S$ and a map $\tau: S\to N$ that is constant along the orbits (i.e., $\tau(yg)=\tau(y)$).
A groupoid bundle is called {\bf principal} if $\tau$ is a surjective submersion and the action is free and transitive along the $\tau$-fibers. 
In this case, %$\rho$ is a free and proper action, 
$\tau$ induces a diffeomorphism $S/G\stackrel{\sim}{\to} N$, and the fibers of $\tau$ are diffeomorphic to $s$-fibers of $G$. Principal bundles for left actions are defined similarly. (When dealing with non-Hausdorff Lie groupoids, one may need to consider principal bundles over non-Hausdorff manifolds.) 
%in which case the action $\theta$ is only principal.

\begin{examples}
\begin{enumerate}[(i)]
\item When $G$ is a Lie group, we recover the classical notion of a principal $G$-bundle. 
%$q:P\to N$, and the local trivializations can be recovered from local sections of $q$; 
%
\item 
Given $G\toto M$, the target map $t:G\to M$ is a $G$-principal bundle, where the action is by right multiplication; this is the {\it universal $G$-bundle}. 
%(in the sense that any other principal $G$-bundle is a pullback of it via a fraction). \textcolor{blue}{is this saying that any principal bundle can be pulled back to a trivial one?} 
%
\item 
If a morphism $\phi: G\to G'$ induces an isomorphism of Lie algebroids, then its kernel $K\toto M$ is an embedded group bundle with discrete fibers, which acts on $G$ by right multiplication making $G\to G'$ into a principal $K$-bundle. Conversely, if $G$ is a Lie groupoid and $K\subset G$ is an embedded group bundle with discrete fibers, then $G/K\toto M$ is a Lie groupoid and $G\to G/K$ yields an isomorphism of their Lie algebroids, see e.g. \cite{dhl}.
%If $K\subset G\toto M$ is a wide embedded subgroupoid then right multiplication $G\curvearrowleft K$, $(g,k)\mapsto gk$ is a principal action \cite{dhl}. 
\end{enumerate}
\end{examples}

Let $P\stackrel{\tau}{\to} N$ be a (right) principal $G$-bundle with moment map $\mu: P\to M$ given by a submersion. Generalizing Example~\ref{ex:gauge}, the {\bf gauge groupoid} of $P$
is the quotient of the submersion groupoid 
$P\times_M P$ by the diagonal $G$-action, denoted by 
$$
P\times_G P := (P\times_M P)/G \toto N.
$$ 
(In fact $P\times_M P$ is a principal $G$-bundle over $P\times_G P$, see \cite[Lem.~5.35]{mmbook}.)
%$G(\theta)=(E\times_M^G E\toto E)$ can be defined as the quotient of the submersion groupoid $E\times_M E\toto E$ of the moment map by the action of $G$. Note that the action $(\theta,\theta):G \action E\times_M E$ is also free and proper. This extends the construction described in the case of a Lie group. 
There is a canonical left action of $P\times_G P$ on $P$, with moment map $\tau: P\to N$, given by $[(p_1,p_2)] p = p_1g$, where $g\in G$ is the only arrow such that $p_2 g= p$.

%right action $E\raction G(\theta)$ given by $e \cdot[e',e'']= g\cdot e'$, where $g\in G$ is the only arrow such that $g\cdot e''=e$.

% In fact, the gauge construction (cf. \ref{gauge})
% of the underlying bundle $G\action P\to M'$ naturally acts on $P$ on the right, by the formula 
% It is not hard to check that this action is smooth, free, and proper for it defines the same relation as the submersion $P\to M$. Thus there is a principal bibundle

% bibundles - principal bibundles

Given Lie groupoids $G'\toto M'$ and $G\toto M$, a $(G',G)$-{\bf bibundle} consists of commuting left and right actions, 
$$\begin{matrix}
G' & \action & P & \curvearrowleft & G \\
\downdownarrows & \swarrow & & \searrow & \downdownarrows\\
M' & & & & M,
\end{matrix}
$$
such that the moment map for one action is invariant with respect to the other. 
%Given $G\toto M$ and $G'\toto M'$ Lie groupoids, a
%{\bf bibundle} consists of commuting left and right actions $\theta:G\action P \raction %G':\theta'$ such that  the moment map of one is invariant for the other:
Such a bibundle is {\bf principal} if both left and right bundles are principal; it is called left or right principal if so is the corresponding bundle.

In a principle bibundle $G'\action P \raction G$, there are canonical identifications of $G$ and $G'$ with gauge groupoids,
$$
G' \cong P\times_G P, \quad G\cong P\times_{G'} P,
$$
compatible with the respective actions on $P$.

%\begin{proposition}
%If $\theta:G\action P \raction G':\theta'$ is a principal bibundle, then there is a canonical isomorphism $G'\cong G(\theta)$ between $G'$ and the gauge groupoid of $\theta$ compatible with the actions.
%\end{proposition}

% Lie algebroid actions, representations and fibrations

\begin{remark}
There is a parallel theory of actions and representations for Lie algebroids, with fruitful interactions with the theory of Lie grou\-poids via differentiation and integration, see e.g. \cite{bcdh,macbook,momr}.
\end{remark}

%%%%%%%%%%%%%%%%%%%%%%%%%%%%%%%%%%%%%%%%%%%%%%%%%%%%%%%%

\section{Multiplicative structures} \label{sec:mult}

In various contexts in differential geometry, Lie groupoids come equipped with geometrical structures suitably compatible with the groupoid multiplication, referred to as {\em multiplicative}. Some examples will be briefly discussed here.

The simplest multiplicative objects on a Lie groupoid $G\toto M$ are the {\bf multiplicative functions}, namely morphisms $f: G \to \R$, with $\R$ viewed as an abelian group, 
$$
f(hg)= f(h) + f(g).
$$

A vector field $X \in \mathfrak{X}(G)$ is {\bf multiplicative} if it defines a morphism $X: G\to TG$. This means that $X|_M$ is a vector field on $M$, $X$ projects onto $X|_M$ via $s$ and $t$, and 
$$
dm_{(h,g)}(X_h,X_g)=X_{hg}
$$ 
for every pair of composable arrows. Alternatively, viewed as a linear function  $T^*G\stackrel{X}{\to} \R$, it is a multiplicative function on the cotangent groupoid.
Multiplicative vector fields are infinitesimal counterparts of flows by groupoid automorphisms \cite{mx98}.

Just as the space of vector fields on a manifold forms a Lie algebra, multiplicative vector fields $\mathfrak{X}_{mult}(G)$ on a Lie groupoid fit into a (strict) {\em Lie 2-algebra}, given by the Lie algebra crossed module defined by the Lie algebra maps
\begin{enumerate}[\ \ $\bullet$]
\item $\Gamma(A_G)\to \mathfrak{X}_{mult}(G)$, $u\mapsto u^r-u^l$, and
\item $\mathfrak{X}_{mult}(G)\stackrel{D}{\to} \mathrm{Der}(\Gamma(A_G))$, $D_X(u)=[X, u^r]|_M$,
\end{enumerate}
where $u^r$ and $u^l$ denote the right- and left-invari\-ant vector fields defined by $u\in \Gamma(A_G)$, see \cite{el,ow}. More generally, multiplicative multivector fields on a Lie groupoid fit into a {\em graded Lie 2-algebra}, a categorification of the graded Lie algebra of multivector fields on a manifold defined by the Schouten bracket \cite{bclx}.

Given $G\toto M$ a Lie groupoid, a differential form $\omega\in\Omega^k(G)$ is {\bf multiplicative} if 
$$
m^*\omega=\pr_1^*\omega+\pr_2^*\omega\in \Omega^k(G\times_MG),
$$
where $G\times_MG$ is the manifold of pairs of composable arrows. If $\omega$ is multiplicative, then so is $d\omega$. Multiplicative 0-forms are multiplicative functions as defined before. Multiplicative 1-forms are the same as sections 
$G\to T^*G$ that are morphisms or, alternatively, linear functions on $TG$ that are multiplicative. A 2-form $\omega$ on $G$ is multiplicative if and only if $\omega^\flat: TG\to T^*G$, $\omega^\flat(X)=i_X\omega$, is a groupoid morphism.

The study of multiplicative 2-forms and bivector fields has been largely stimulated by Poisson geometry, especially with the advent of symplectic and Poisson groupoids \cite{cdw,we88}.

A {\bf symplectic groupoid}  is a Lie groupoid $G\toto M$ with a symplectic form $\omega\in\Omega^2(G)$ that is multiplicative. In this case, being multiplicative is equivalent to asking that the graph of the groupoid multiplication be a lagrangian submanifold of
$(G,\omega)\times (G,\omega)\times (G,-\omega)$. 
If $(G\toto M,\omega)$ is a symplectic groupoid, then $\dim(G) = 2 \dim(M)$, the submanifold of units $M\subseteq G$ is lagrangian, the inversion map is anti-symplectic, and source fibers are symplectic orthogonals to target fibers. The link with Poisson geometry comes from the fact that $M$ inherits a Poisson structure for which the target map is a Poisson map. Moreover, there is an identification of the Lie algebroid $A_G$ with the cotangent algebroid $T^*M$ defined by this Poisson structure. In this sense, a symplectic groupoid is always an integration of a Poisson manifold. Conversely, whenever the cotangent algebroid of a Poisson manifold is integrable, its source-simply-connected integration is a symplectic grou\-poid \cite{mx00}, see also \cite{catfel}.
\begin{example}
For a Lie groupoid $G\toto M$, the canonical symplectic structure on its cotangent groupoid $T^*G \toto A_G^*$ makes it into a symplectic groupoid. The corresponding Poisson structure on $A_G^*$ is the one dual to the Lie algebroid  $A_G$.
\end{example}

The theory of actions of Lie groupoids can be enhanced to symplectic groupoids.
Given a  symplectic groupoid $(G\toto M, \omega)$,  a {\bf hamiltonian $G$-space} is a symplectic manifold $(S,\omega)$ equipped with an action $\rho: G\times_M S\to S$ satisfying 
$$
\rho^*\omega_S = \pr_1^*\omega+\pr_2^*\omega_S.
$$ 
When $G$ is a Lie group, hamiltonian spaces for $T^*G\toto \g^*$ recover usual hamiltonian spaces in symplectic geometry.
Various aspects of the theory of hamiltonian actions carry over to actions of symplectic groupoids \cite{miwe}. 

Symplectic groupoids admit important generalizations, with motivations in mathematical physics. 
In one direction, {\em quasi-symplectic grou\-poids}  (aka {\em (twisted) presymplectic grou\-poids}) \cite{bcwz,xu} are Lie groupoids equipped with a multiplicative 2-form that may fail to be closed or nondegenerate in a way controlled by the groupoid structure, to be recalled in $\S$ \ref{sec:nerve}.
Extending the correspondence between symplectic groupoids and Poisson structures, quasi-sym\-plectic grou\-poids arise as integrations of {\em (twis\-ted) Dirac structures} \cite{bcwz}. Quasi-symplectic grou\-poids also arise in the context of {\em shifted symplectic structures} \cite{safr}, and their actions express important generalized notions of hamiltonian actions.

%see Example~\ref{ex:quasi}. %Extending the correspondence between symplectic groupoids and Poisson structures, (twisted) presym\-plectic grou\-poids arise as integrations of (twisted) Dirac structures \cite{bcwz}. 
In another direction, {\bf Poisson grou\-poids} \cite{we88} are Lie groupoids $G\toto M$ equip\-ped with a Poisson structure $\pi \in \mathfrak{X}^2(G)$ that is multiplicative, in the sense that the contraction map $T^*G\to TG$ is a groupoid morphism.
Poisson groupoids are simultaneous generalizations of symplectic groupoids and Poisson-Lie groups. Their infinitesimal counterparts are {\em Lie bialgebroids} \cite{mx00}, which include Poisson manifolds and Lie bialgebras as special cases. The integration of Poisson groupoids (as Poisson manifolds) lead to {\em symplectic double groupoids}, see e.g. \cite{mac}.

A treatment of general multiplicative tensor fields on Lie groupoids and their infinitesimal counterparts can be found in \cite{budr}

\section{Proper Lie groupoids}

% proper groupoids

%In this section, we discuss proper groupoids in more detail. 
We briefly discuss linearization and rigidity properties of proper Lie groupoids.
A (Hausdorff) Lie groupoid $G\toto M$ is {\bf proper} if $(t,s): G\to M\times M$ is a proper map, or equivalently, if given $A, B \subset M$ compact sets, the arrows from $A$ to $B$ form a compact set. In a proper Lie groupoid, the isotropy groups $G_x$ are compact, the orbit space $M/G$ is Hausdorff, and the orbits $O\subset M$ are closed and embedded.
An {\bf $s$-proper} groupoid is a Lie groupoid $G\toto M$ whose source map is proper. Clearly, an $s$-proper Lie groupoid must be proper.

% examples

\begin{examples}
\begin{enumerate}[(i)]
\item A Lie group $K\toto *$ is pro\-per if and only if it is compact. A manifold $M\toto M$ is always a proper groupoid.
\item If $\tau: S\to M$ is a surjective submersion then the submersion groupoid $S\times_M S\toto S$ is proper. Conversely, if $G\toto M$ is a free, proper groupoid  then $M/G$ is a manifold such that $M\to M/G$ is a submersion, and
 $G\toto M$ is isomorphic to the corresponding submersion groupoid.
 % submersion groupoid arising from $M\to M/G$.
\item An action groupoid is proper if and only if so is the action. In particular, any action of a compact group yields a proper groupoid.
\item 
If a foliation $F\subseteq TM$ has compact leaves and Hausdorff orbit space, then its holonomy groups are finite \cite{epstein}, and by Reeb stability, the holonomy groupoid is $s$-proper. The converse also holds: if $\mathrm{Hol}(F)\toto M$ is $s$-proper, then the leaves are compact, and the orbit space is Hausdorff.
\end{enumerate}
\end{examples}

% orbits

Let $G\toto M$ be a Lie groupoid, $O\subset M$ an orbit, and denote by $G_O\toto O$ the restriction of $G$ to the orbit. There is an inclusion 
$$
(TG_O\toto TO)\to(TG|_{G_O}\toto TM|_O)
$$
of VB-groupoids, and the quotient $NG_O\toto NO$, formed by the normal bundles, is a VB-groupoid called the {\bf linear model} of $G\toto M$ around $O$. The linear model has trivial core, hence (see Prop.~\ref{prop:rep}) it encodes a representation, called the {\bf normal representation} $(G_O\toto O)\action (NO\to O)$. 
%Given $y\xfrom g x$ in $G_O$, the normal representation $\rho_g:N_xO\to N_yO$ is induced by any bisection passing through $g$.
The linear model still makes sense if we replace an orbit $O$ with any invariant submanifold $S\subset M$. The linearization problem, asking if there exists an identification of  the linear model with an honest neighborhood, has a positive answer for proper groupoids. 

\begin{theorem}[Weinstein-Zung Linearization]
Given $G\toto M$ a proper Lie groupoid  and $O\subset M$ an orbit (or invariant submanifold), there are open subsets $U \subset M$ and $V\subset NO$ containing $O$ %$O\subset U\subset M$ and $O\subset V\subset NO$, 
and a Lie groupoid isomorphism
$$(G_U\toto U)\cong(G_V\toto V).
$$
If $G$ is $s$-proper, then $U$ and $V$ can be taken to be invariant.
\end{theorem}

% corollaries

The previous theorem extends several classical linearization results, e.g.
Ehresmann’s theorem (for proper submersions), the tube theorem (for actions of compact Lie groups), and Reeb stability (for compact, Hausdorff foliations). 

%\begin{enumerate}[$\bullet$]
%    \item Ehresmann Theorem for proper submersions $q:E\to M$;
%    \item Reeb stability Theorem for compact Hausdorff foliations $F\subset TM$; 
%    \item Tube Theorem for actions of compact Lie groups $K\action M$.
%\end{enumerate}

The original approach to the linearization theorem reduces the problem to the fixed point case by considering a small transverse to the orbit \cite{we02}, and uses averaging arguments to prove linearization around  a fixed point \cite{zung}. Subsequent proofs use averaging arguments to either obtain a multiplicative vector field allowing a Moser trick \cite{crst}, or construct  compatible Riemannian structures on $G$ and $M$ so that the linearization is achieved by the exponential maps \cite{dhfe}.

%The original approach to the linearization theorem \cite{zung} consists in reducing the problem to the fixed point case by considering a small transverse $T$ to the orbit \cite{we02}, and using averaging arguments to obtain an isomorphism between $G_T\toto T$ and the action groupoid of the normal representation $G_x\ltimes N_xO\toto N_xO$. Subsequent proofs use averaging arguments to either construct a multiplicative vector field allowing a Moser trick \cite{crst}, or construct  compatible Riemannian structures on $G$ and $M$ so that the linearization is achieved by the exponential maps \cite{dhfe}.

% rigidity

A consequence of the linearization result is the rigidity of compact Lie groupoids \cite{crmst,dhf2}. A {\bf deformation} of a Lie groupoid $G\toto M$    consists of  deformations $(s_\epsilon,t_\epsilon,m_\epsilon)$ of the source, target, and multiplication maps so that at each time $\epsilon$ they define a Lie groupoid, and two deformations are equivalent if they differ by a groupoid-theoretic isotopy. %they define a Lie groupoid can be defined as a Lie groupoid $\phi:(G\times I\toto M\times I)\to (I\toto I)$ such that $\phi_0$ is the projection and $G\toto M$ identifies with the fiber at $0$. Two deformations $\phi,\phi'$ are {\bf equivalent} if, after shrinking $I$, there is a fibered isomorphism between them. Given a deformation, $M\times 0$ is an invariant manifold, and the linear model around $G\times 0\toto M\times 0$ identifies with the constant deformation.

\begin{corollary}
{A compact Lie groupoid $G\toto M$ is {\em rigid}, i.e., any of its deformations is equivalent to the constant deformation.}    
\end{corollary}

%A cohomological theory of Lie groupoids deformations was started in \cite{crmst}, and a proof of this corollary using Riemannian metrics is given in \cite{dhf2}.

% Haar densities

Averaging arguments play a key role in 
the linearization and rigidity results and rely on the notion of {\em Haar systems}.  Recall that a {\bf density} on a vector bundle $E\to M$ is a map $\sigma:\mathrm{Fr}(E)\to\R$ such that $\sigma(A\beta)=|\det(A)|\sigma(\beta)$ for every frame $\beta\in \mathrm{Fr}(E)_x$. Densities define a trivial line bundle;
%even in the degenerate case $E=0_M$, for $\mathrm{Fr}(0)=\{\emptyset\}$, 
when $E$ is orientable, densities correspond to sections of $\wedge^{\rm top}E^*\to M$. Given a map $f:N\to M$ and a density $\sigma$ on $E\to M$, there is a {\bf pullback density}  $f^*\sigma$ on $f^*E\to N$. A density $\sigma$ on a manifold $M$ is a density on its tangent bundle, and it is {\bf normalized} if $\int_M \sigma=1$. 

\begin{proposition}
A proper Lie groupoid $G\toto M$ admits a density $\sigma$ on $A_G\to M$ such that for every $x\in M$ its pullback $\sigma^x$ on $T (s^{-1}(x))=t^*A$ is normalized and has compact support.
\end{proposition}

Such a density $\sigma$ is called a {\bf Haar density} on $G\toto M$, and the family $\{\sigma^x\}$ is called a {\bf Haar system}. 
%\textcolor{blue}{properties omitted}
%It has the following properties:
%\begin{enumerate}[H1)]
%    \item $I(f)(x)=\int_{G(-,x)}f(g)\mu^x(g)$ 
%    is smooth for every $f\in C^\infty(G)$;
% \item $\int_{G(-,y)} f(gh)\mu^y(g)= \int_{G(-,x)}f(g)\mu^x(g)$ for any $y\xfrom h x$ in $G$ and $f\in C^\infty(G(-,x))$; and 
% \item ${\rm supp}(\mu^x)$ is compact and 
%$\int_{G(-,x)}\mu^x(g)=1$.
%\end{enumerate}

\begin{examples}
\begin{enumerate}[(i)]
    \item A Haar density on a compact Lie group is a usual (right-invariant) Haar  measure, and the only Haar density on a manifold $M\toto M$ is 1.
    \item A Haar density on a submersion groupoid is a vertical density, inducing integration along the fibers. In the particular case of the Cech groupoid of an open cover $\U$ of $M$, this is the same as a partition of unity subordinated to $\U$.
\end{enumerate}
\end{examples}

% examples of results proving using haar densities

A noteworthy application of Haar densities is the proof that any proper Lie groupoid admits a compatible real analytic structure \cite{martinez}, which can be regarded as a variant of Hilbert's fifth problem. In the next sections, we will discuss some further applications of Haar densities, such as the Morita invariance of (homotopy) representations, and vanishing theorems for cohomology.

%To finish this section, we mention one of the most notable results on proper Lie groupoids, whose proof also involves averaging methods, and it can be regarded as a solution to Hilbert's fifth problem. It was proven in \cite{martinez}.

%\begin{theorem}
%Proper Lie groupoids are real analytic.
%\end{theorem}

%%%%%%%%%%%%%%%%%%%%%%%%%%%%%%%%%%%%%

\section{Morita equivalence} \label{sec:morita}

% motivation

%Lie groupoids serve as models for differentiable stacks.
%The orbit space $M/G$ of a Lie groupoid is not in general a smooth manifold, it is a singular quotient. The idea behind stacks is to regard $G\toto M$ as a device to perform differential geometry in the singular quotient $M/G$.

From a differential geometric standpoint, {\em Mo\-rita equivalence} is a way of identifying Lie grou\-poids with the same ``transverse geometry''.

% Morita morphisms

A Lie groupoid morphism $\phi:(G'\toto M')\to (G\toto M)$ is a {\bf Mo\-rita morphism} if 
\begin{enumerate}[\ M1)]
\item it induces a homeomorphism $M'/G'\to M/G$ between the orbit spaces,
\item it gives an isomorphism $G'_{x'}\to G_{\phi(x')}$ between the isotropy groups for every $x'$, 
\item $d\phi: N_{x'}O\to N_{\phi(x')}O$ is a linear isomorphism between the normal spaces for every $x'$.
\end{enumerate}

%\blue{same as preserving normal reps...}

%
If $\phi$ is a groupoid isomorphism, or more generally, if $\phi$ is invertible up to homotopy, then $\phi$ is Morita, but the converse does not hold. 
A fibration $\phi$ is Morita if and only if its kernel is identified with the submersion groupoid of $\phi_0$.
% 

% examples

\begin{examples}
    \begin{enumerate}[(i)]
\item %Given $\{U_i\}$ an open cover of a manifold $M$, the projection 
%$$\bigg(\coprod_{ji} U_{ji}\toto\coprod_i U_i\bigg)\to (M\toto M)$$ 
%from the Cech groupoid is Morita. More generally, 
Given $\tau: M'\to M$ a surjective submersion, the projection 
$$(M'\times_M M'\toto M')\to (M\toto M)$$ 
from the submersion groupoid is a Morita fibration, and it is invertible up to homotopy if and only if $\tau$ admits a global section.
%Conversely, a Lie groupoid is Morita equivalent to a unit groupoid if and only if it is a submersion groupoid.
%        
\item Given $G\toto M$ a transitive Lie groupoid and $x\in M$, the inclusion of the isotropy
$$(G_x\toto \{x\})\to (G\toto M)$$ is Morita, and it is invertible up to homotopy if and only if the associated principal $G$-bundle $P\to M$ is trivial.
%A Lie groupoid is Morita equivalent to a Lie group if and only if it is transitive.
%
\item 
Given $Q$ an orbifold and $\U$, $\U'$  orbifold atlases, a {\em refinement map} $\lambda:\U\to\U'$, given by a collection of orbifold chart embeddings, 
%$\lambda_i:(U_i,G_i,\phi_i)\to(U'_{j(i)},G'_{j(i)},\phi_{j(i)})$, 
yields a morphism $\lambda_*:G_\U\to G_{\U'}$ between the Cech groupoids (Example \ref{ex:orbifold}) that is Morita.
\end{enumerate}
\end{examples}

% equivalent formulations

Morita morphisms can be equivalently descri\-bed as fully faithful and essentially surjective maps \cite[$\S$ 4]{dH}, also known as {\em weak equivalences} \cite[$\S$ 5.4]{mmbook}, 
%or as morphisms related to pullback groupoids, 

\begin{comment}that a Lie groupoid morphism $\phi:(G'\toto M')\to (G\toto M)$ is {\bf essentially surjective} if $G\times_M M'\to M$, $(y\xfrom g \phi(x'),x')\mapsto y$, is a surjective submersion, and it is {\bf fully faithful} if the following is a cartesian square:
    $$\xymatrix{
G' \ar[r]^\phi \ar[d]_{(t,s)} & G \ar[d]^{(t,s)} \\
M'\times M'\ar[r]^\phi & M\times M.}$$
Recall also that, given $G\toto M$ a Lie groupoid, and $f:M'\to M$ a map that is transverse to the orbits, the {\bf pullback groupoid} $f^*(G)\toto M'$ is defined, where $f^*(G)$ is the pullback of $(t,s):G\to M\times M$ along $f\times f$. 

\begin{proposition}[\cite{dH}]
The following are equivalent:
\begin{enumerate}[(i)]
    \item $\phi$ is Morita;
    \item $\phi$ is fully faithful and essentially surjective;
    \item $\phi_0$ meets every orbit transversely and $\phi_1:G'\to\phi_0^*(G)$ is an isomorphism with the pullback.
\end{enumerate}
\end{proposition}
\end{comment}
% Morita equivalence
Two Lie groupoids $G, G'$ are said to be {\bf Morita equivalent} if there are Morita morphisms 
$$
G\from \tilde G\to G'.
$$ 
Morita equivalence is clearly a reflexive and symmetric relation.
To see that it is transitive, and hence an equivalence relation, it is convenient to work with {\em homotopy fibered products} \cite{mmbook,dH}.

Given Lie groupoid morphisms $\phi:G\to H$ and $\psi:G'\to H$, 
their {\bf homotopy fibered product} $G \tilde{\times}_H G'$ is the (set-theoretic) groupoid where an object is a triple 
$$
(x,\phi(x)\xfrom h\psi(x'),x')\in M\times H\times M'
$$
and an arrow from $(x, h, x')$ to $(y, h', y')$ is a pair $(y\xfrom g x,y'\xfrom{g'}x')\in G\times G'$ inducing a commuta\-tive square in $H$:
$$
 \phi(g) h = h' \psi(g').
$$

{The homotopy fibered product $G \tilde{\times}_H G'$ comes with natural projections $\pi_1$ onto $G$, $\pi_2$ onto $G'$, 
%$\pi_1: G\tilde{\times}_H G'\to G$ and $\pi_2:G\tilde{\times}_H G'\to G'$ 
and a homotopy $\phi\pi_1\cong\psi\pi_2$, defining a commuting up to homotopy square that is universal:
$$\xymatrix{
%K \ar@/^/[rrd] \ar@/_/[ddr] & & \\
  G\tilde{\times}_H G'\ar[r]^{\pi_2} \ar[d]_{\pi_1} \ar@{}[dr]|{\cong}& G' \ar[d]^{\psi}\\
 G \ar[r]_{\phi} & H.
}$$

\begin{proposition}
If $\psi$ is a Morita morphism, then $G\tilde{\times}_H G'$ is a well-defined Lie groupoid, and $\pi_1$ is a Morita fibration.    
\end{proposition}

Now if $G\from H\to G'$ and $G'\from H'\to G''$ are Morita equivalences, their composition can be defined by means of the homotopy fibered product of $H\to G'\from H'$:
$$\xymatrix@R=10pt@C=10pt{
& & H\tilde\times_{G'}H' \ar[dl] \ar[dr] & & \\
& H \ar[dl] \ar[dr] & & H' \ar[dl] \ar[dr] & \\
G & & G' & & G''.
}$$

A {\bf differentiable stack} is the Morita equivalence class of a Lie groupoid $G\toto M$ \cite{bx}, denoted by  $[G]$. The geometry of the stack $[G]$ is expressed by the properties and constructions on $G\toto M$ that are invariant under Morita equivalence.  For example, if $\phi: G'\to G$ is a Morita morphism, then $G'$ is proper if and only if so is $G$. The stacks corresponding to proper groupoids are called {\bf separated} or {\bf Hausdorff}.
It follows from the linearization theorem that a Hausdorff stack is locally modeled by a linear representation of a compact group.

\begin{examples}
\begin{enumerate}[(i)]
\item By regarding manifolds as unit grou\-po\-ids we obtain a one-to-one correspondence between manifolds and Hausdorff stacks without isotropy. 
\item The Cech groupoid construction sets a one-to-one correspondence between orbifolds and Hausdorff stacks with finite isotropy. 
\end{enumerate}
\end{examples}

%
%If $K\action M$ is a proper action, $x\in M$, $x\in S\subset M$ is a {\em slice} around $x$, then the inclusion $(G_S\toto S)\to (G_U\toto U)$ is Morita, where $U$ is the saturation of $S$. Since $G_S\toto S$ is isomorphic to the action groupoid of the normal representation $G_x\action N_xO$, it follows that $K\ltimes M$ is locally Morita equivalent to a representation of a compact group. It can be seen that a Lie groupoid is locally Morita equivalent to a representation of a compact group if and only if it is proper, and this is a key theorem on the theory.
%\item Given $F\subset TM$ a foliation and $S\subset M$ a complete transversal, then $Hol(F)_S\toto S$ is a well-defined etale groupoid, which completely encodes the transverse geometry of the foliation. A Lie groupoid is Morita equivalent to an etale groupoid if and only if it has discrete isotropies. If $F$ has compact leaves and Hausdorff orbit space, then its holonomy groups are finite \cite{}, $Ho(F)\toto M$ is a proper groupoid, and $M/Hol(F)$ is an {\em orbifold}.

% manifolds and stacks

%\begin{example}
%\begin{enumerate}[(i)]
%\item Smooth manifolds identify with Hausdorff stacks without isotropy;
%\item Orbifolds identify with Hausdorff stacks with discrete isotropy groups. In fact, if $G\toto M$ is proper and it has finite isotropy, it follows from the linearization theorem that the orbit space admits orbifold charts around every point.
%\end{enumerate}
%\end{example}

% maps

A {\bf stacky map} $\psi/\phi:[G]\to [H]$ is given by a fraction of morphisms $G\xfrom{\phi} \tilde G\xto\psi H$ where $\phi$ is Morita, and two fractions define the same stacky map, $\psi/\phi=\psi'/\phi'$, if there is a third fraction $\phi''/\psi''$ fitting into a commutative diagram up to homotopy:
$$\xymatrix@R=10pt{
 & \tilde G \ar[ld]_{\phi} \ar[rd]^{\psi} & \\
G  & \tilde G'' \ar[u] \ar[d] \ar[r]^{\psi''} \ar[l]_{\phi''} & H. \\
& \tilde G' \ar[ul]^{\phi'} \ar[ur]_{\psi'} & }$$

Homotopy fibered products can be used to show that (i) the above relation on fractions is indeed an equivalence relation, so stacky maps are well-defined, (ii) any stacky map can be realized by a fraction $\psi/\phi$ with $\phi$ a Morita fibration, and (iii) there is a well-defined composition of stacky maps, leading to a well-defined category of differentiable stacks.

% cocycle

Given a Lie groupoid  $G\toto M$  and a submersion $\tau:\U=\coprod_i U_i\to M$ defined by an open cover, denote by $G_\U=\tau^*G$ the pullback groupoid (generalizing the Cech groupoid of $\U$). Since any surjective submersion admits local sections, it is always possible to realize a stacky map by a fraction
%Given $G\xfrom\phi \tilde G\xto\psi H$ a fraction, by the previous lemma, we can suppose that $\phi$ is a Morita fibration, so $\phi_0:\tilde M\to M$ is a surjective submersion.  If $\sigma_i:U_i\to \tilde M$ are local sections for $\phi_0$ defined over an open cover, the $\sigma_i$ yield a Morita morphism $G_\U\to \tilde G$, proving that any stacky map $[G]\to[H]$ can be realized by a fraction
$$G\from G_\U\to H$$
for some open cover of $M$. This observation leads to a cocycle description of stacky maps \cite{dH}.

\begin{example}
Let $M$ be a manifold and $K$ be a Lie group. We will describe stacky maps $[M]\to [K]$. For an open cover $\U$ of $M$ with Cech groupoid $G_\U$, the
fraction 
$$
M\from G_\U\to K
$$ 
is the same as a $K$-cocycle supported on $\U$, and two such fractions are equivalent if and only if the {cocycles} define the same {principal $K$-bundle}. Thus, stacky maps $[M]\to [K]$ are in one-to-one correspondence with principal $K$-bundle over $M$. The stack $[K]$ is called the {\bf classifying stack} of $K$, and it can be regarded as a finite-dimensional model for the classifying space $BK$.
\end{example}

The previous example can be generalized to the case when $M$ and $K$ are replaced by arbitrary Lie groupoids $G$ and $H$.
The isomorphism class of a right-principal bibundle 
$$
\begin{matrix}
G & \action & P & \curvearrowleft & H \\
\downdownarrows & \swarrow & & \searrow & \downdownarrows\\
M & & & & N
\end{matrix}
$$
is called a  {\bf Hilsum-Skandalis map} $[P]:G\dasharrow H$ \cite{hs}.
Given a right-principal bibundle, one can build the action grou\-poid $G\ltimes P\rtimes H$ and obtain a fraction 
$$
G\from G\ltimes P\rtimes H\to H.
$$
See e.g.\cite[$\S$ 2.5]{mmbook2} for the next result.

\begin{proposition}
The above construction sets a one-to-one correspondence between Hilsum-Skan\-dalis maps $[P]:G\dasharrow H$ and stacky maps $[G]\to[H]$. In particular, $G$ and $H$ are Morita equivalent if and only if they fit into a principal bibundle $G\action P \raction H$.
\end{proposition}

% VB-Morita

The category of representations $\mathrm{Rep}(G)$ of a Lie groupoid $G$ is a Morita invariant. More generally, we consider the {\em derived category} $\mathrm{VB}[G]$, with VB-groupoids $\Gamma\to G$ as objects and homotopy classes of maps over $\id_G$ as arrows. (There is an embedding $\mathrm{Rep}(G)\to \mathrm{VB}[G]$ coming from the identification of representations with VB-grou\-poids with  trivial core, Thm.~\ref{thm:ruthVB}.)
A key fact \cite{dho} is that a VB-groupoid morphism
$\Gamma\to \Gamma'$ covering $\id_G:G\to G$ is Morita if and only if it yields a quasi-isomorphism of the corresponding core complexes, if and only if it is invertible up to homotopy. The following is proven in \cite{dho}.

%Lie groupoid representations and, more generally, the (derived) category of VB-groupoids, are Morita invariants.
%Let $\Phi:\Gamma\to \Gamma'$ be a VB-groupoid morphism covering $\phi:G\to G'$. %It turns out that $\Phi$ is Morita if and only if it is so on the base and on the fibers \cite{dho}. In particular, 
%When $G'=G$ and $\phi=\id_G$, $\Phi$ is Morita if and only if it yields a quasi-isomorphism of the corresponding core sequences, if and only if it is invertible up to homotopy, see \cite{dho}. The {\em derived category} $\mathrm{VB}[G]$ has VB-groupoids $\Gamma\to G$ as objects and homotopy classes of maps over $\id_G$ as arrows. By identifying representations with VB-groupoids with  trivial core \blue{cross ref}, we obtain an  embedding $\mathrm{Rep}(G)\to \mathrm{VB}[G]$ of the category of $G$-representations. See \cite{dho} for the next result.

\begin{theorem}
If $\phi:G'\to G$ is a Morita morphism, then $\phi^*:\mathrm{VB}[G]\to \mathrm{VB}[G']$ is an equivalence of categories, extending the equivalence $\phi^*:\mathrm{Rep}(G)\to \mathrm{Rep}(G')$.
\end{theorem}

Objects in $\mathrm{VB}[G]$ can be regarded as {\em stacky vector bundles}, and examples include the tangent and cotangent stacks $[TG]$ and $[T^*G]$.  The Lie 2-algebra of multiplicative vector fields (see $\S$ \ref{sec:mult}) is well defined, up to homotopy equivalence, for the stack $[G]$, and encodes the sections of $[TG]\to [G]$ (i.e., vector fields on $[G])$ \cite{el,ow} (see also \cite{bclx}).%, see also \cite{xu-etc}.

%\begin{theorem}
%If $\phi:G'\to G$ is Morita, then $\phi$ induces an %equivalence between the Lie 2-algebras of multiplicative vector fields of $G'$ and $G$.
%\end{theorem}

%%%%%%%%%%%%%%%%%%%%%%%%%%%%%%%%%%%%

\section{Nerve and cohomology}\label{sec:nerve}

% nerve

%Denote by $\Delta$ the category of finite ordinals $\Delta[n]=\{0,1,\dots,n\}$ and order-preserving maps. A {\bf simplicial set} is a contravariant functor $X:\Delta^\circ\to\text{Sets}$, and can be described by the collection $(X_n,d_i,s_j)$, where $d_i=\delta_i^*:X_n\to X_{n-1}$ and $s_j=\sigma_j^*:X_n\to X_{n+1}$.

To define cohomology of Lie groupoids, we first introduce their {\em nerve}.
We refer to \cite{gj} for the basics on simplicial objects. 

The {\bf nerve} of a Lie groupoid $G\toto M$   is the simplicial manifold $N(G)=(G_n,d_i,s_j)$ with $G_0=M$, $G_1=G$, and $G_n$ given by
the chains of $n$ composable arrows, 
$$
x_n\xfrom{g_n}x_{n-1}\xfrom{g_{n-1}}\dots\xfrom{g_2} x_1\xfrom{g_1}x_0.
$$
By transversality, each $G_n$ is an embedded submanifold of $G^n$. The {\em face maps} $d_i:G_n\to G_{n-1}$ erases $x_i$, composing $g_{i+1}$ and $g_i$ if $0<i<n$, and discarding the first or last arrow if $i=0,n$. The {\em degeneracy maps} $s_j:G_n\to G_{n+1}$ repeats $x_j$ and inserts an identity. One can associate a topological space $|N(G)|$ to $N(G)$ by means of its {\em geometric realization}, which is a model for the classifying space $BG$ of principal $G$-bundles.

% nerve2

%More conceptually, an $n$-simplex $g\in NG_n$ can be regarded as a functor $[n]\to G$ from the $n$-th ordinal, and faces and degeneracies are given by pre-composition with the corresponding ordinal map. 

% examples

%\begin{examples}
%\begin{enumerate}[i)]
%\item The nerve of a unit grou\-poid is a constant simplicial manifold $M_n=M$, where $d_i=s_j=\id_M$;
%\item  If $K\toto\ast$ is a discrete group, then its nerve 
%$(K^n,d_i,s_j)$
%recovers the Milnor's construction codifying the classifying space $BK$.
%\item The nerve of an action groupoid $K\ltimes M\toto M$ serves as a finite-dimensional model for the Borel construction used in equivariant cohomology.
%\end{enumerate}
%\end{examples}

% horn-fillings

The nerve yields a fully faithful functor 
$$
\text{Lie groupoids}\xto{\text{Nerve}}\text{Simplicial Manifolds}
$$
whose essential image can be described by a ``horn-filling'' whose generalization leads to {\em higher Lie groupoids}, see e.g. \cite{duskin,henriques}. %\blue{def horn, weak m-groupoid etc commented out}. 

\begin{comment}
\blue{We could skip: Recall that $\Delta[n]$ is the simplicial set represented by $[n]\in \Delta$, and $\Lambda^k[n]\subset\Delta[n]$ is the {\bf horn}, spanned by every face of $\id_n\in\Delta[n]_n$ except  $d_k(\id_n)$. 
Denote by $d_{n,k}:X_n=\hom(\Delta[n], X)\to \hom(\Lambda^k[n], X)$ to the restriction map.
A simplicial manifold is a {\bf weak Lie $m$-groupoid} if it satisfies the following \cite{duskin,henriques}: 
\begin{enumerate}[W1)]
\item $d_{n,k}$ is a surjective submersion for all $n,k$, and
\item $d_{n,k}$ is a diffeomorphism for all $n> m$.
\end{enumerate}
\begin{proposition}[Grothendieck]
$X$ is a weak Lie 1-groupoid if and only if it is isomorphic to the nerve of a Lie groupoid $G\toto M$.
\end{proposition}}
\end{comment}

% algebra of functions

For a Lie groupoid $G\toto M$, its algebra of {\bf differentiable cochains} $(C(G),\delta,\cup)$ is the differential graded algebra of functions on the nerve:
\begin{enumerate}[\ \ $\bullet$]
\item $C(G)=\bigoplus_{p\geq0}C^\infty(G_p)$,
\item $\delta=\sum_{i}(-1)^id_i^*:C^p(G)\to C^{p+1}(G)$,
\item $(f_1\cup f_2)(g)=(-1)^{pq}f_1(t_qg)f_2(s_pg)$, where $g\in G_{p+q}$, $f_1\in C^q(G)$, $f_2\in C^p(G)$.
\end{enumerate}
Here use the notation $s_pg\in G_p$ and $t_qg\in G_q$ for the chains formed by the first $p$ arrows and the last $q$ arrows of $g \in G_n$, respectively.
The {\bf differentiable cohomology} of $G$, denoted by $H_{d}(G)$, is the cohomology of $C(G)$. Elements in $H_d^0(G)$ are the basic functions on $M$, i.e., those which are constant along the orbits, and $H_d^1(G)$ consists of multiplicative functions on $G$ modulo homotopy (i.e., modulo functions of the form $t^*f-s^*f$, for $f\in C^\infty(M)$). 

% E-valued cochains
Given a vector bundle $E\to M$, an {\bf $E$-valued cochain} is a section 
$c\in C^p(G,E)=\Gamma(G_p,t_0^*E)$, and it is {\bf normalized} if it vanishes on degenerate simplices. The space $C(G,E)=\oplus_p C^p(G,E)$ has a canonical $C(G)$-module structure, given by 
$$(c\cdot f)(g)=(-1)^{pq}c(t_qg)f(s_pg)$$
where $g\in G_{p+q}$, $c\in C^q(G,E)$ and $f\in C^p(G)$.
See \cite[$\S$ 2.3]{arcr} for the next result.

\begin{proposition}
A representation $G\action E$ is equivalent to a degree 1 differential $D$ in $C(G,E)$ preserving the normalized cochains and satisfying the Leibniz rule, 
$$
D(c\cdot f)=D(c)\cdot f+(-1)^{q}c\cdot\delta(f),
$$ 
where $c\in C^q(G,E)$ and $f\in C^p(G)$.   
\end{proposition}

% ruth

\begin{comment}
Consider now $E= \bigoplus_{n=0}^{m} E_n$ a graded vector bundle over $G_0$. An $E$-valued cochain is a section 
$c\in C^p(G,E) =\bigoplus_{i-j=p}\Gamma(G_i,t_0^*E_j)$. 
A {\bf homotopy representation} $R:G\action E$ is a degree 1 differential $D$ on 
$$C(G,E)=\bigoplus_pC^p(G,E)$$
preserving the normalized cochains and satisfying Leibniz \cite{arcr}. The bi-degree can be used to write $D=\sum_{m\geq 0}D_m$, where each $D_m$ %:C^i(G,E_j)\to C^{i+m}(G,E_{j+m-1})$ 
define and is defined by a map $R_m:s_0^*E_j\to t_0^*E_{j+m-1}$ via the following formulas, where $c\in C^i(G,E_j)$ and $g\in G_{m+i}$:
\begin{align*}
d_m(c)(g)  = &(-1)^{j}R_m^{t_mg}c(s_ig) && m\neq 1\\
d_1(c)(g)  = &(-1)^{j} R_1^{t_1g}c(d_{i+1}g) +\\ &+\sum_{r=0}^i (-1)^{i+j+r+1} c(d_{r}g)
\end{align*}

The equation $D^2=0$ implies that $(E^x,R^x_0)$ is a chain complex for every $x$, $R^g_1:(E^x,R^x_0)\to (E^y,R^y_0)$ is a quasi-isomorphism for every $y\xfrom g x$, and $R^{h,g}_2:R^{hg}_1\then R^h_1R^g_1$ is a chain homotopy for every pair $(h,g)$. When $E=E_0$ this recovers the definition of representation, and when $E=E_0\oplus E_1$, this recovers the representations up to homotopy from Section \ref{section:actions}. 
\end{comment}

This characterization of representations can be directly adapted to the case when $E= \oplus_{k=0}^{m} E_k$ is a graded vector bundle over $M$, and leads to the notion of {\bf representations up to homotopy} of $G$ on $E$  as a degree 1 differential in the space of $E$-valued cochains \cite{arcr}, extending the 2-term case discussed in $\S$ \ref{section:actions}. The {\bf cohomology with coefficients} $H_d(G,E)$ is the cohomology of $(C(G,E),D)$. In a proper groupoid, it is possible to define homotopy-type operators for cochains by averaging with a Haar system to prove the following vanishing theorem \cite{arcr,crainic}.

\begin{theorem}%[Vanishing Theorem] 
Let $E=\bigoplus_{k=0}^m E_k\to M$ be a graded vector bundle with a representation up to homotopy of a proper Lie groupoid $G$. Then $H^n(G,E)=0$ for every $n>m$.
\end{theorem}

% Morita equivalence

If $\phi: G'\to G$ is a morphism and $G\action E$ is a representation up to homotopy, then there is a {\em pullback representation} $\phi^*E$ of $G'$
%$\phi^*(\rho):G \action \phi_0^*E$ 
and an {\em induced morphism} $\phi^*:H_d(G,E)\to H_d(G',\phi^*E)$. The next result shows that differentiable cohomology is a Morita invariant \cite{crainic,dhos}.

\begin{theorem}%[Morita invariance]
If $G\action E$ is a representation (up to homotopy) of $G$ and
$\phi:G'\to G$ is a Morita morphism, then $\phi^*:H_d(G,E)\to H_d(G',\phi^*E)$ is an isomorphism.
\end{theorem}

% Bott-Shulman complex and Morita invariance

A variant of the cohomology of a Lie groupoid that combines the smooth and simplicial structures goes as follows. The {\bf de Rham complex} of $G\toto M$  is the double complex $(\Omega^q(G_p),\delta,d_{dR})$ with $\delta(\omega)=\sum_{i}(-1)^id_i^*(\omega)$ and $d_{dR}$ the de Rham differential: 
$$%\Omega(G)=\begin{cases}
\xymatrix@R=10pt@C=10pt{
\vdots & \vdots & \vdots & \\
\Omega^2(M) \ar[r] \ar[u] & \Omega^2(G_1) \ar[r] \ar[u] & \Omega^2(G_2) \ar[r] \ar[u] &  \dots\\
\Omega^1(M) \ar[r] \ar[u] & \Omega^1(G_1) \ar[r] \ar[u] & \Omega^1(G_2) \ar[r] \ar[u] &  \dots\\
C^\infty(M) \ar[r]^\delta \ar[u]^{d_{dR}} & C^\infty(G_1) \ar[r] \ar[u] & C^\infty(G_2) \ar[r] \ar[u] &  \dots
}%\end{cases}
$$
The cohomology of the corresponding total complex is the {\bf de Rham cohomology} $H_{dR}(G)$ and agrees with the cohomology of the classifying space $BG$ \cite{bss}.
%\blue{inf model?}

\begin{examples}
\begin{enumerate}[i)]
\item The de Rham cohomology of $M\toto M$ is that of $M$.
\item If $G_\U$ is the Cech groupoid of an open cover in $M$, then its de Rham complex is known as the {\em Cech-de Rham complex}, and $H_{dR}(G_\U)\cong H_{dR}(M)$, see \cite{botu}.
%\item If $K$ is a Lie group, its de Rham cohomology computes cohomology of the classifying space, $H_{dR}(K)=H(BK)$;
\item When $G=K\ltimes M$ is an action groupoid we recover the equivariant cohomology of $K\action M$, $H_{dR}(G)=H_K(M)$.
\end{enumerate}
\end{examples}

% final comments

The de Rham cohomology $H_{dR}(G)$ is also a Morita invariant, see 
\cite{behrend}. Cohomology of Lie groupoids have counterparts for Lie algebroids, with Van-Est-type theorems relating them \cite{arcr11,crainic}.

Multiplicative forms on $G$ ($\S$ \ref{sec:mult}) have a simple interpretation in the de Rham complex as elements of bi-degrees $(1,k)$ that are closed with respect to the horizontal differential $\delta$. In particular, a multiplicative $k$-form $\omega$ with $d_{dR}\omega=0$ defines a degree $(k+1)$ cocycle in the total de Rham complex. A way to relax the closedness of $\omega$ is by considering more general cocycles $(\eta,\omega)$, where $\eta\in \Omega^{k+1}(M)$ and $\omega\in \Omega^k(G)$; the cocycle condition in this case becomes $d_{dR}\eta=0$, $\delta \omega=0$ (i.e., $\omega$ is multiplicative) and $d_{dR}\omega = t^*\eta-s^*\eta$. 

Such cocycles yield generalizations of symplectic groupoids ($\S$ \ref{sec:mult}).
%\begin{example}\label{ex:quasi}
A {\em $\eta$-twisted symplectic grou\-poid} is a Lie groupoid $G\toto M$ equip\-ped with a multiplicative and nondegenerate 2-form $\omega \in \Omega^2(G)$, and a closed 3-form $\eta \in \Omega^3(M)$ such that $d_{dR}\omega = t^*\eta-s^*\eta$. By relaxing the nondegeneracy of $\omega$ in a suitable, controlled way, one obtains {\em quasi-sym\-plec\-tic groupoids} (aka {\em $\eta$-twisted presymplectic grou\-poids}) \cite{bcwz,xu}; as shown in \cite[$\S$ 5.2]{dho} this ``weak nondegeneracy'' amounts  to requiring that $\omega^\flat: TG \to T^*G$ be a Morita morphism, rather than an isomorphism; equivalently, the induced map of core complexes 
$$
(A\stackrel{\rho}{\to} TM) \to (T^*M \stackrel{\rho^*}{\to} A^*)$$
 is just required to be a quasi-isomorphism. 

%\end{example}

%\blue{mention connection with shifted symplectic? }

%. looking at non-homogeneous forms. A {\bf twisted} closed multiplicative $k$-form $(\eta,\omega)$ consists of $\eta\in \Omega^{k+1}(G_0)$ and $\omega\in\Omega^k(G_1)$ such that $\eta$ is closed, $\omega$ is multiplicative, and $d_{dR}(\omega)=t^*\eta-s^*\eta$, or in other words, such that $(\eta,\omega)$ is closed for the total differential. Twisted closed multiplicative forms appear for instance when integrating twisted Dirac manifolds, see \cite{bcwz}.

%Multiplicative forms in the Bott-Shulman complex???
%(Pre)symplectic groupoids???

%Deformation cohomology???

%Bott spectral sequence???

%\section{Complementary topics and further readings}

%Index theorems??

%%%%%%%%%%%%%%%%%%%%%%%%%%%%%%%%%%%%%%%%%%%%%%%%%%%%%%%%%%%%%%%%%%%

%%%%%%%%%%%%%%%%%%%%%%%%%%%%%%%%%%%%%%%%%%%%%%%%%%%%%%%%%%%

%%%%%%%%%%%%%%%%%%%%%%%%%%%%%%%%%%%%%%%%%%%%%%%%%%%%%%%%%%%

%%%%%%%%%%%%%%%%%%%%%%%%%%%%%%%%%%%%%%%%%%%%%%%%%%%%%%%%%%%%%%%%%%%%%%%%%%%%
\end{document}